\documentclass[a4paper,fleqn]{cas-sc}

\usepackage[numbers]{natbib}

\usepackage{amssymb,amsmath,amsfonts}
\usepackage{enumerate}
\usepackage{pdfpages}
\usepackage{mathrsfs}
\usepackage{graphicx}
\usepackage{tabularx}
\usepackage{xcolor}
\usepackage{comment}

\newcommand{\MyParagraph}[1]{\noindent\textbf{#1.}\;}
\newcommand{\MyCite}[1]{\cite{#1}}
\newtheorem{Theorem}{Theorem}

\newtheorem{Proposition}{Proposition}
\newtheorem{Corollary}{Corollary}
\newtheorem{Remark}{Remark}
\newcommand{\ee}{{\rm e}}
\newcommand{\dd}{{\rm d}}

\newcommand{\pa}{\partial}
\newcommand{\pat}{\pa_{t}}

\newcommand{\ve}[2]{\begin{pmatrix} #1 \\ #2 \end{pmatrix}}

\newcommand{\vect}[2]{\left({{#1}\atop{#2}}\right)}
\renewcommand{\ve}[3]{\begin{pmatrix} #1 \\ #2 \\ #3 \end{pmatrix}}
\newcommand{\velo}{\vec{\mathrm{v}}}
\newcommand{\magind}{\vec{\mathrm{b}}}
\newcommand{\nO}{\mathcal{O}}

\begin{document}
\let\WriteBookmarks\relax
\def\floatpagepagefraction{1}
\def\textpagefraction{.001}
\shorttitle{Operator splitting methods for the Bloch-Torrey equation}
\shortauthors{Kaltenbacher, Mu\~noz, Thalhammer}

\title [mode = title]{Operator splitting methods for non-autonomous evolution equations comprising three parts, with application to the Bloch-Torrey equation of MRI}                      



\author[1]{Barbara Kaltenbacher}[
                        orcid=0000-0002-3295-6977]
\cormark[1]
\fnmark[1]
\ead{barbara.kaltenbacher@aau.at}

\credit{Theorems, Proofs, Model, Writing}

\author[1]{Pablo Mu\~noz}[
]
\fnmark[1]
\ead{pablo.munoz@aau.at}

\credit{Theorems, Proofs, Implementation, Numerical Tests, Writing}

\affiliation[1]{organization={Department of Mathematics, University of Klagenfurt},
                addressline={Universit\"atsstr 65-67}, 
                city={Klagenfurt},
                postcode={9020}, 
                country={Austria}}
\author[2]{Mechthild Thalhammer}[
                        orcid=0000-0001-5958-6219]
\fnmark[2]
\ead{mechthild.thalhammer@uibk.ac.at}

\credit{Theorems, Proofs, Methodology, Writing}

\affiliation[2]{organization={Leopold--Franzens-Universit{\"a}t Innsbruck, Institut f{\"u}r Mathematik},
                addressline={Technikerstr. 13 / VII}, 
                postcode={6020}, 
                city={Innsbruck},
                country={Austria}}

\cortext[cor1]{Corresponding author}
\fntext[fn1]{The work by PM and BK has been funded by the Austrian Science Fund (FWF) 
[10.55776/F100800]}
\fntext[fn2]{
The work by MT has been funded by the Austrian Science Fund (FWF)
[10.55776/PAT1281625]
}


\begin{abstract}
    In this paper we devise and analyze Lie Trotter splitting schemes for non-autonomous evolution equations, following two approaches: 1) A fully non-autonomous scheme, motivated by the potential use of explicit solutions of some subproblems. 2) A splitting scheme relying on the concept of Magnus integrators, where we freeze the operators at the beginning of each time interval as this suffices for first order convergence of the Lie-Trotter integrator. 
We derive first order convergence in a general setting and apply the analysis to the Bloch-Torrey equation of magnetic resonance imaging, where a decomposition into three subproblems is well-motivated by the physical role of the underlying operators.
Numerical tests, complementing the time stepping scheme by a finite element space discretization, illustrate the theoretical findings.
\end{abstract}




\maketitle

\section{Introduction}

A vast amount of physical, biological or engineering processes are modeled by time dependent partial differential equations PDEs, where the various effects are described by a sum of terms driving the evolution and the presence of some time dependent control term renders the system non-autonomous.
The work in this paper is particularly motivated by 
model based quantitative magnetic resonance imaging \cite{MRI-uniqueness,ScholandWangRoeloffsRosenzweigUecker2023,Uecker_Review:2021,UeckerHohageBlockFrahm2008} which amounts to reconstructing space dependent coefficients in a space-time PDE -- the Bloch-Torrey equation.
Its efficient numerical solution is therefore essential for solving the corresponding inverse problem.
On the other hand, steering the evolution described by the Bloch-Torrey equation by means of time dependent control functions is crucial for enhancing image quality.
We are therefore faced with a non-autonomous evolutionary system with a special structure that lends itself into a decomposition into three parts corresponding to three different physical effects.

Hence, this paper aims to provide a framework that allows to analyze operator splitting schemes in evolution equations with decomposition into more than the conventionally considered two parts, and to make use of it for devising and analyzing three-splitting schemes for the Bloch-Torrey equation.


\MyParagraph{Problem setting}
The fundamental PDE model describing the physics underlying magnetic resonance imaging is the 
Bloch-Torrey equation for the magnetization $\vec M$,
see, e.g., \cite{bloch46,torrey56,Nishimura2010}.
A convective version is 
\begin{equation}\label{eqn:bloch-torrey}
\begin{aligned}
    \frac{d}{dt} \vec M(t, \vec r) = \gamma \vec M(t, \vec r) \times  
\magind(t, \vec r)
    -  \begin{pmatrix} R_2(\vec r)\,M_{\mathsf{x}}(t, \vec r) \\ R_2(\vec r)\,M_{\mathsf{y}}(t, \vec r) \\ R_1(\vec r)\,(M_{\mathsf{z}}(t, \vec r)-M^{eq}(\vec r)) \end{pmatrix}
    +\nabla\cdot\Bigl(D(\vec r) \, \nabla \vec M(t, \vec r)\Bigr) 
    -(\velo(t,\vec r)\cdot\nabla)\vec M(t,\vec r).
\end{aligned}
\end{equation}

Quantities appearing in \eqref{eqn:bloch-torrey} are the magnetization $\vec M$, the magnetic induction $\magind$, the equilibrium magnetization $M^{eq}$, the spin-lattice and spin-spin relaxation rates $R_1$, $R_2$, the velocity $\velo$ and the diffusion $D$.

The transversal part of the magnetic flux density typically takes the form 
\[
\magind_\perp (t, \vec r) :=
\vect{\magind_{\mathsf{x}}}{\magind_{\mathsf{y}}}
(t, \vec r) =\underline{c}^+ (\vec r) \, p(t)
\]
with 
the radio frequency RF field $p$ and $\underline{c}^+$ describing spatial dependency of the transmitted field as determined by the scanner hardware.
\\
Its longitudinal part is of the form
$\magind_{\mathsf{z}}(t, \vec r) =B^0 + \delta B^0(\vec r) +   \vec r \cdot \vec G(t)$ with the static magnetic field $B^0$, the local field inhomogeneity $\delta B^0$ and the field gradient $\vec G$.

In \eqref{eqn:bloch-torrey}, the first term on the right hand side describes \textit{rotation} \footnote{more precisely, a combination of precession due by the main magnetic field $B^0$ and rotation away from the longitudinal $z$ direction due to the radio frequency pulse} driven by an imposed magnetic field with flux density $\magind(t, \vec r)$, and the second one \textit{relaxation} towards the equilibrium magnetization, that is conventionally assumed to be aligned in $z$-direction. 
\\
\textit{Diffusion} -- which distinguishes the Bloch-Torrey PDE from the Bloch ODE model -- is possibly anisotropic but typically acts the same way on all components of the magnetization 
\[
\begin{aligned}
\nabla\cdot\Bigl(D \, \nabla \vec M)&:=\Bigl(\nabla\cdot\Bigl(D \, \nabla M_j)\Bigr)_{j\in\{1,2,3\}}
\end{aligned}
\]
with a symmetric positive definite diffusion tensor $D:\mathbb{R}^3\mapsto \mathbb{R}^{3\times 3}$.
It is therefore to be understood as being applied component wise whenever acting on a vector valued function.
\\
The \textit{motion} term $(\velo(t,\vec r)\cdot\nabla)\vec M(t,\vec r)$ allows to describe mechanical displacements (e.g., due to patient motion or blood flow) inside the imaged object.

Partition of the magnetization vector $\vec M$ into a longitudinal $M_z$ and a transversal $M_{\perp}=\vect{M_x}{M_y}$ component and considering it in a frame rotating at the Larmor frequency $\gamma B^0$ to remove the (known) high frequency component of the transversal magnetization, together with the abbreviation
\[
\vect{a_x}{a_y}^\vdash = \vect{a_y}{-a_x}
\]
leads to the coupled transversal-longitudinal system (skipping the space variable $\vec r$ for brevity)
\begin{equation}\label{eqn:bloch-torrey_1}
\begin{aligned}
\frac{d}{dt} M_\perp(t) &= \gamma \magind_z(t)\, M_\perp(t)^\vdash -\gamma M_z(t)\, \magind_\perp(t)^\vdash- R_2\, M_\perp(t) 
+[\nabla\cdot(D \nabla \cdot) -(\velo(t)\cdot\nabla)]\,M_\perp(t)\\
\frac{d}{dt} M_z(t) &= \gamma \magind_\perp(t)^\vdash\cdot M_\perp(t) - R_1\, (M_z(t)-M^{eq})
+[\nabla\cdot(D \nabla \cdot) -(\velo(t)\cdot\nabla)]\,M_z(t)\\
\end{aligned}
\end{equation}
In order to end up with a homogeneous system, we introduce $M^{eq}$ as an additional auxiliary state $M_e(t)$ that satisfies $\frac{d}{dt} M_e(t)=0$ with initial condition $M_e(0)=M^{eq}$.

We consider two decompositions into three parts (both could be externded to decompositions into four parts by separating diffusion and motion)

\MyParagraph{Decomposition~I}
\begin{equation*}
\begin{aligned}
&A(t)\ve{v_\perp}{v_z}{v_e} = \gamma\ve{\magind_z(t)\, v_\perp^\vdash - v_z\, \magind_\perp(t)^\vdash}{\magind_\perp(t)^\vdash\cdot v_\perp}{0}
&&\text{rotation} \\
&B\ve{v_\perp}{v_z}{v_e} = -\ve{R_2\, v_\perp}{R_1(v_z-v_e)}{0}
&&\text{relaxation} \\
&C(t)\ve{v_\perp}{v_z}{v_e} = \ve{[\nabla\cdot(D \nabla \cdot) -(\velo(t)\cdot\nabla)]}{[\nabla\cdot(D \nabla \cdot) -(\velo(t)\cdot\nabla)]}{0}
&&\text{diffusion \& motion}.
\end{aligned}
\end{equation*}
Note that $B$ is autonomous; so is $C$ in case of stationary (or vanishing) motion $\velo$.

\MyParagraph{Subproblems of Decomposition~I}\\ 
\textit{rotation:}\ 
$\frac{d}{dt} u(t)=A(t)u(t)$ is a pointwise in space non-autonomous ODE that can be solved by a standard numerical ODE scheme.

\noindent
\textit{relaxation:}\
$\frac{d}{dt} u(t)=Bu(t)$ has the simple explicit solution
\[
\begin{aligned}
&u_\perp(t)=\exp(-R_2(\vec r)(t-t_0))u_\perp(t_0)\\
&u_z(t)-u_e(t_0)=\exp(-R_1(\vec r)(t-t_0))(u_z(t_0)-u_e(t_0))\\
&u_e(t)=u_e(t_0)
\end{aligned}
\]

\noindent
\textit{diffusion \& motion:}\
$\frac{d}{dt} u(t)=C(t)u(t)$ is a standard advection-diffusion problem, for which well-known numerical schemes exist.

\MyParagraph{Decomposition~II}
\begin{equation*}
\begin{aligned}
&A(t)\ve{v_\perp}{v_z}{v_e} = \gamma\ve{- v_z\, \magind_\perp(t)^\vdash}{\magind_\perp(t)^\vdash\cdot v_\perp}{0}
&&\text{rotation - coupling}  \\
&B(t)\ve{v_\perp}{v_z}{v_e} = \ve{\gamma\,\magind_z(t)\, v_\perp^\vdash -R_2\, v_\perp}{-R_1(v_z-v_e)}{0}
&&\text{rest-of-rotation \& relaxation}  \\
&C(t)\ve{v_\perp}{v_z}{v_e} = \ve{[\nabla\cdot(D \nabla \cdot) -(\velo(t)\cdot\nabla)]}{[\nabla\cdot(D \nabla \cdot) -(\velo(t)\cdot\nabla)]}{0}
&&\text{diffusion \& motion}.
\end{aligned}
\end{equation*}

\MyParagraph{Subproblems of Decomposition~II}\\ 
\textit{rotation - coupling:}\
In the quite common case $\magind_\perp(t)=c^+(\vec r) p(t)$, with $c^+(\vec r)\in\mathbb{R}^2$, $p(t)\in\mathbb{R}$, an explicit solution 
of $\frac{d}{dt} u(t)=A(t)u(t)$ is given by
\begin{equation}\label{explicitsolution_BT_DII_A}
\begin{aligned}
&u(t,\vec r)=\bigl(\exp(\int_{t_0}^tA(s)\, ds)u_0\bigr)(\vec r)\\
&=\,U(\vec r)\,\text{diag}\Bigl(1,\,\exp(\imath\gamma|c^+(\vec r)|\int_{t_0}^tp(s)\, ds),\,\exp(-\imath\gamma|c^+(\vec r)|\int_{t_0}^tp(s)\, ds)\Bigr)\,U(\vec r)^*u_0,
\end{aligned}
\end{equation}
with 
\[
U(\vec r)= \frac{1}{\sqrt{2}|c|}\left(\begin{matrix}
c_x^+(\vec r)&\imath c_y^+(\vec r)&\imath c_y^+(\vec r)\\
c_y^+(\vec r)&-\imath c_x^+(\vec r)&-\imath c_x^+(\vec r)\\
0&|c^+(\vec r)|&|c^+(\vec r)|
\end{matrix}\right), \quad 
|c|^2=c_x^2+c_y^2
\]
as can be readily checked.

\noindent
\textit{rest-of-rotation \& relaxation:}\
With $\magind_{\mathsf{z}}(t, \vec r) = \delta B^0(\vec r) +   \vec r \cdot \vec G(t)$ in the rotating frame, and the notation
\begin{equation}\label{kt}
{\vec k}(t)=\frac{\gamma}{2\pi}\int_0^t \vec G(\tau)\, d\tau 
\qquad\text{\ldots k-space trajectory},
\end{equation}
we have that 
\[
\begin{aligned}
&u_\perp(t)=\exp(-R_2(\vec r)(t-t_0))
\vect{
\cos(\omega(\vec r,t))u_x(t_0)-\sin(\omega(\vec r,t))u_y(t_0)}{
\sin(\omega(\vec r,t))u_x(t_0)+\cos(\omega(\vec r,t))u_y(t_0)}\\
&u_z(t)=\exp(-R_1(\vec r)(t-t_0))u_z(t_0)\\
&\text{with }\omega(\vec r,t)=\gamma \delta B_0(\vec r)(t-t_0)+{\vec r}\cdot{\vec k}(t)
\end{aligned}
\]
solves $\frac{d}{dt} u(t)=B(t)u(t)$. 

\noindent
\textit{diffusion \& motion:}\
remains as in decomposition I.

\medskip

We will study \eqref{eqn:bloch-torrey_1} and its discretzation in more detail in Sections~\ref{sec:BT}, \ref{sec:Numerics}.

\medskip 

More generally, the splitting approaches devised and analyzed in this paper apply to evolutions that are subject to a time dependent control via a bilinear control term. This motivates us to also study a more abstract setting involving a non-autonomous bilinear control term alongside an advection-diffusion component, cf. Section~\ref{sec:MultDiffAdv}.

\medskip

The remainder of the paper is organized as follows. 
Section~\ref{sec:GeneralFramework} provides a general framework for the analysis of the local and global time discretization error in a Lie-Trotter splitting method involving three non-autonomous components.
In Section~\ref{sec:MultDiffAdv}, we derive the evolution estimates and commutator estimates required for this error analysis for a general system comprising two multiplication operators and an advection-diffusion operator.
Section~\ref{sec:BT} focuses on the Bloch-Torrey equation and Section~\ref{sec:Numerics} provides some further details on its space discretzation with a priori error estimates, as well as numerical experiments illustating the theoretical findings.



\section{General framework}
\label{sec:GeneralFramework}

\MyParagraph{Main objective}
In the following, we provide a guide line for the rigorous convergence analysis of the first-order Lie--Trotter splitting method adapted to the time integration of linear partial differential equations comprising three (or more) defining operators of different types and explicit time dependencies. 
For this purpose, we consider a general analytical framework of non-autonomous linear evolutionary problems with three parts. 
Our considerations will then be made concrete for the Bloch-Torrey equation and confirmed by numerical examples in the subsequent sections.
Essential theoretical means for the derivation of local and global error bounds for operator splitting methods and related Magnus-type integrators applied to linear and nonlinear evolution equations are given in our previous works, see for instance~\MyCite{AuzingerKochThalhammer2015,BlanesCasasGonzalezThalhammer2025,BlanesCasasThalhammer2018,GonzalezOstermannThalhammer2006,KaltenbacherNikolicThalhammer2015} and references given therein; 
see also the general references \cite{McLachlanQuispel2002,Blanes_Casas_Murua_2024} on operator splitting methods as well as 
\cite{BatkaiCsomosFarkasNickel2011} and the references therein for the non-autonomous setting. 

\MyParagraph{Non-autonomous evolutionary problems}
Throughout, we study the initial value problem for a linear evolution equation of the form 
\begin{equation}
\label{eq:IVP}
\begin{cases}
u'(t) = F(t) \, u(t) = A(t) \, u(t) + B(t) \, u(t) + C(t) \, u(t)\,, \quad t \in [t_0, T]\,, \\
u(t_0) \, \text{ given}\,, 
\end{cases}
\end{equation}
with solution $u: [t_0, T] \to X$ taking values in the underlying Banach space.
Regarding the theoretical analysis of time integration methods, it is convenient to introduce the associated evolution operator
\begin{equation}
\label{eq:EF}
u(t) = E_F(t - t_0; t_0) \, u(t_0)\,, \quad t \in [t_0, T]\,, \quad E_F(0; t_0) = I\,, 
\end{equation}
indicating the dependencies on the defining operator, the time increment, and the initial time. 
Hereby, $I: X \to X$ denotes the identity operator. 

\MyParagraph{Time stepping approach}
For the numerical approximation of~\eqref{eq:IVP}, we employ a standard time stepping approach. 
That is, for certain time grid points with corresponding increments 
\begin{equation*}
t_0 < t_1 < \dots < t_N = T\,, \quad \tau_n = t_{n+1} - t_n\,, \quad n \in \{0, 1, \dots, N-1\}\,, 
\end{equation*}
and an initial approximation~$u_0$, we determine numerical solution values through a recurrence of the form 
\begin{equation*}
u_{n+1} = S_F(\tau_n; t_n) \, u_n \; \approx \; u(t_{n+1}) = E_F(\tau_n; t_n) \, u(t_n)\,, \quad n \in \{0, 1, \dots, N-1\}\,.
\end{equation*}

\MyParagraph{Approach 1 (Reformulation, Lie--Trotter splitting)}
A natural approach to define the numerical evolution operator relies on the reformulation of~\eqref{eq:IVP} as autonomous problem by introducing an auxiliary scalar function 
\begin{equation*}
\begin{cases}
u'(t) = F\big(s(t)\big) \, u(t)\,, \quad 
s(t) = t\,, \quad s'(t) = 1\,, \quad t \in [t_0, T]\,, \\
s(t_0) = t_0\,, \quad u(t_0) \, \text{ given}\,, \\
\end{cases}
\end{equation*}
In accordance with the first-order Lie--Trotter splitting method, we consider the associated subproblems on a temporal subinterval such that 
\begin{equation*}
s(t) = s(t_n) + \int_{t_n}^{t} s'(\sigma) \; \dd \sigma\,, \quad t \in [t_n, t_n + \tau_n]\,,
\end{equation*}
and compose their solutions 
\begin{equation*}
\begin{aligned}
&\begin{cases}
v_1'(t) = A\big(s_1(t)\big) \, v_1(t)\,, 
\quad s_1'(t) = \tfrac{1}{3}\,, \quad t \in [t_n, t_n + \tau_n]\,, \\
s_1(t_n) = t_n\,, \quad v_1(t_n) = u_n\,, \\
\end{cases} \\
&\begin{cases}
v_2'(t) = B\big(s_2(t)\big) \, v_2(t)\,, 
\quad s_2'(t) = \tfrac{1}{3}\,, \quad t \in [t_n, t_n + \tau_n]\,, \\
s_2(t_n) = s_1(t_n + \tau_n) = t_n + \tfrac{1}{3} \, \tau_n\,, \quad v_2(t_n) = v_1(t_n + \tau_n)\,, \\
\end{cases} \\
&\begin{cases}
v_3'(t) = C\big(s_3(t)\big) \, v_3(t)\,, 
\quad s_3'(t) = \tfrac{1}{3}\,, \quad t \in [t_n, t_n + \tau_n]\,, \\
s_3(t_n) = s_2(t_n + \tau_n) = t_n + \tfrac{2}{3} \, \tau_n\,, \quad v_3(t_n) = v_2(t_n + \tau_n)\,.
\end{cases} 
\end{aligned}
\end{equation*}
In this way, we obtain the new numerical solution value through 
\begin{equation*}
u_{n+1} =  v_3(t_n + \tau_n)\,, \quad n \in \{0, 1, \dots, N-1\}\,.
\end{equation*}
It is notable that the resolution of the time components 
\begin{equation*}
\begin{gathered}
s_1(t) = s_1(t_n) + \tfrac{1}{3} \, (t - t_n) = t_n + \tfrac{1}{3} \, (t - t_n)\,, \\
s_2(t) = s_2(t_n) + \tfrac{1}{3} \, (t - t_n) = t_n + \tfrac{1}{3} \, \tau_n + \tfrac{1}{3} \, (t - t_n)\,, \\
s_3(t) = s_3(t_n) + \tfrac{1}{3} \, (t - t_n) = t_n + \tfrac{2}{3} \, \tau_n + \tfrac{1}{3} \, (t - t_n)\,, \\
t \in [t_n, t_n + \tau_n]\,,
\end{gathered}
\end{equation*}
as well as resubstitution yields a non-autonomous formulation of the subproblems 
\begin{equation*}
\begin{aligned}
&\begin{cases}
v_1'(t) = A\big(t_n + \tfrac{1}{3} \, (t - t_n)\big) \, v_1(t)\,, \quad v_1(t_n) = u_n\,, \\
v_1(t) = E_A\big(t - t_n; t_n; \tfrac{1}{3}\big) \, v_1(t_n)\,, \quad t \in [t_n, t_n + \tau_n]\,, \\ 
\end{cases} \\
&\begin{cases}
v_2'(t) = B\big(t_n + \tfrac{1}{3} \, \tau_n + \tfrac{1}{3} \, (t - t_n)\big) \, v_2(t)\,, \quad v_2(t_n) = v_1(t_n + \tau_n)\,, \\
v_2(t) = E_B\big(t - t_n; t_n + \tfrac{1}{3} \, \tau_n; \tfrac{1}{3}\big)\,, \quad t \in [t_n, t_n + \tau_n]\,, \\ 
\end{cases} \\
&\begin{cases}
v_3'(t) = C\big(t_n + \tfrac{2}{3} \, \tau_n + \tfrac{1}{3} \, (t - t_n)\big) \, v_3(t)\,, \quad 
v_3(t_n) = v_2(t_n + \tau_n)\,, \\
v_3(t) = E_C\big(t - t_n; t_n + \tfrac{2}{3} \, \tau_n; \tfrac{1}{3}\big)\,, \quad t \in [t_n, t_n + \tau_n]\,. 
\end{cases} \\
\end{aligned}
\end{equation*}
Similarly to~\eqref{eq:EF}, the dynamics of the subproblems is captured by the corresponding evolution operators. 
We indicate the operator that defines the right-hand side, the time increment, the initial time, and additionally the scaling factor occurring in the time dependency. 
Altogether, in the context of the first-order Lie--Trotter splitting method, we obtain the relation 
\begin{equation}
\label{eq:Approach1}
\begin{gathered}
u_{n+1} = S_F(\tau_n; t_n) \, u_n = E_C\big(\tau_n; t_n + \tfrac{2}{3} \, \tau_n; \tfrac{1}{3}\big) \, E_B\big(\tau_n; t_n + \tfrac{1}{3} \, \tau_n; \tfrac{1}{3}\big) \, E_A\big(\tau_n; t_n; \tfrac{1}{3}\big) \, u_n\,, 
\quad
n \in \{0, 1, \dots, N-1\}\,.
\end{gathered}
\end{equation}
We point out that this procedure extends to higher-order operator splitting methods.

\MyParagraph{Approach 2 (Magnus integrator, Lie--Trotter splitting)}
An alternative approach for the treatment of non-autonomous evolutionary problem is based on the concept of Magnus integrators. 
In order to establish a first-order approximation, it suffices to replace the defining operator on each temporal subinterval by its value at the initial time
\begin{equation*}
F(t) \approx F(t_n)\,, \quad t \in [t_n, t_n + \tau_n]\,, \quad n \in \{0, 1, \dots, N-1\}\,. 
\end{equation*}
In this way, we arrive at an autonomous problem  
\begin{equation*}
\begin{cases}
v'(t) = F(t_n) \, v(t) = A(t_n) \, u(t) + B(t_n) \, u(t) + C(t_n) \, u(t)\,, \quad t \in [t_n, t_n + \tau_n]\,, \\
v(t_n) \, \text{ given}\,, 
\end{cases}
\end{equation*}
where the application of the first-order Lie--Trotter splitting method yields the new numerical solution value. 
That is, we compose the solutions to the corresponding time-frozen subproblems 
\begin{equation*}
\begin{gathered}
\begin{cases}
v_1'(t) = A(t_n) \, v_1(t)\,, \quad t \in [t_n, t_n + \tau_n]\,, \\
v_1(t_n) = u_n\,, 
\end{cases} \\
\begin{cases}
v_2'(t) = B(t_n) \, v_2(t)\,, \quad t \in [t_n, t_n + \tau_n]\,, \\
v_2(t_n) = v_1(t_n + \tau_n)\,, 
\end{cases} \\
\begin{cases}
v_3'(t) = C(t_n) \, v_3(t)\,, \quad t \in [t_n, t_n + \tau_n]\,, \\
v_3(t_n) = v_2(t_n + \tau_n)\,, 
\end{cases} \\
\end{gathered}
\end{equation*}
to obtain the defining relation 
\begin{equation*}
u_{n+1} =  v_3(t_n + \tau_n)\,, \quad n \in \{0, 1, \dots, N-1\}\,.
\end{equation*}
Adapting the notation for the evolution operator associated with the time-frozen problem 
\begin{equation*}
v(t) = E_{F(t_n)}(t - t_n) \, v(t_n)\,, \quad t \in [t_n, t_n + \tau_n]\,, 
\end{equation*}
and analogously for the subproblems, we have  
\begin{equation}
\label{eq:Approach2}
\begin{gathered}
u_{n+1} = S_F(\tau_n; t_n) \, u_n = E_{C(t_n)}(\tau_n) \, E_{B(t_n)}(\tau_n) \, E_{A(t_n)}(\tau_n) \, u_n\,, 
\quad
n \in \{0, 1, \dots, N-1\}\,.
\end{gathered}
\end{equation}
Generalisations involve evaluations of the defining right-hand side at appropriately chosen intermediate times and the application of higher-order operator splitting methods. 

\MyParagraph{Remark}
We point out that both approaches are practicable for an application to the Bloch--Torrey equation. 
The first approach makes use of the knowledge of exact solutions to non-autonomous subproblems. 
The second approach is somewhat simpler, since it replaces time-dependent functions by their piecewise constant interpolants. 

\MyParagraph{Global error representation}
As standard, we employ the telescopic identity 
\begin{equation}
\label{eq:GlobalError}
\begin{split}  
&u_N - u(t_N) \\
&= \prod_{n=0}^{N-1} S_F(\tau_n; t_n) \, u_0 - \prod_{n=0}^{N-1} E_F(\tau_n; t_n) \, u(t_0) \\
&= \prod_{n=0}^{N-1} S_F(\tau_n; t_n) \, \big(u_0 - u(t_0)\big) + \Bigg(\prod_{n=0}^{N-1} S_F(\tau_n; t_n) - \prod_{n=0}^{N-1} E_F(\tau_n; t_n)\Bigg) \, u(t_0) \\  
&= \prod_{n=0}^{N-1} S_F(\tau_n; t_n) \, \big(u_0 - u(t_0)\big) 
+ \sum_{n=0}^{N-1} \Bigg(\prod_{m=n+1}^{N-1} S_F(\tau_m; t_m)\Bigg) \big(S_F(\tau_n; t_n) - E_F(\tau_n; t_n)\big) \Bigg(\prod_{m=0}^{n-1} E_F(\tau_m; t_m) \Bigg) \, u(t_0)\,.   
\end{split}  
\end{equation}
also referred to as Lady Windermere's fan, in order to reduce a global error analysis to stability bounds for the exact and numerical evolution operators as well as local error estimates. 
The product of non-commutative operators is defined downwards, and the empty product is given by the identity operator. 

\MyParagraph{Stability}
Stability bounds for multiple compositions of the numerical evolution operators associated with the two approaches 
\begin{equation*}
\begin{aligned}    
S_F(\tau_n; t_n) \, u_n &= E_C\big(\tau_n; t_n + \tfrac{2}{3} \, \tau_n; \tfrac{1}{3}\big) \, E_B\big(\tau_n; t_n + \tfrac{1}{3} \, \tau_n; \tfrac{1}{3}\big) \, E_A\big(\tau_n; t_n; \tfrac{1}{3}\big)\,, \\ 
S_F(\tau_n; t_n) \, u_n &= E_{C(t_n)}(\tau_n) \, E_{B(t_n)}(\tau_n) \, E_{A(t_n)}(\tau_n)\,, 
\end{aligned}
\end{equation*}
see~\eqref{eq:Approach1} and~\eqref{eq:Approach2}, respectively, are provided below. 
We consider the setting, where~$A$ and~$B$ represent space-time-dependent multiplication operators and~$C$ a diffusion-advection operator.
In particular, we specialise the arguments to the Bloch--Torrey equation. 

\MyParagraph{Local error analysis}
Important tools in our local error analysis are integral representations by the variation-of-constants formula and the derivation of suitable expansions of the arising integrands. 
Our main interest is to identify commutators, defined by 
\begin{equation*}
[G, H] = G \, H - H \, G
\end{equation*}
on certain subspaces of the underlying Banach space, since these capture the essential regularity assumptions on the prescribed initial state and hence the exact solution to~\eqref{eq:IVP}. 
Fixing an index $n \in \{0, 1, \dots, N-1\}$, we next detail this strategy for our first approach, 

\MyParagraph{Integral representation (Approach 1)}
On each temporal subinterval, we consider the local error operator as a time-dependent quantity
\begin{equation*}
L_F(t - t_n; t_n) = S_F(t - t_n; t_n) - E_F(t - t_n; t_n)\,, \quad t \in [t_n, t_n + \tau_n]\,.
\end{equation*}
With regard to the evolutionary problem 
\begin{equation*}
\begin{cases}
\tfrac{\dd}{\dd t} \, E_F(t - t_n; t_n) = F(t) \, E_F(t - t_n; t_n)\,, \quad t \in [t_n, t_n + \tau_n]\,, \\
E_F(0; t_n) = I\,, 
\end{cases}
\end{equation*}
we cast the initial value problem for the numerical evolution operator into the form 
\begin{equation*}
\begin{cases}
\tfrac{\dd}{\dd t} \, S_F(t - t_n; t_n) = F(t) \, S_F(t - t_n; t_n) + R(t - t_n; t_n)\,, \quad t \in [t_n, t_n + \tau_n]\,, \\
S_F(0; t_n) = I\, 
\end{cases}
\end{equation*}
with the so-called defect $R$.
We conclude at once that the difference satisfies 
\begin{equation*}
\begin{cases}
\tfrac{\dd}{\dd t} \, L_F(t - t_n; t_n) = F(t) \, L_F(t - t_n; t_n) + R(t - t_n; t_n)\,, \quad t \in [t_n, t_n + \tau_n]\,, \\
L_F(0; t_n) = 0\,.
\end{cases}
\end{equation*}
As a consequence, we obtain the integral representation 
\begin{equation*}
L_F(t - t_n; t_n) = \int_{0}^{t - t_n} E_F(t - t_n - \zeta; t_n) \, R(\zeta; t_n) \; \dd \zeta\,, \quad t \in [t_n, t_n + \tau_n]\,.
\end{equation*}
Relying on a framework that ensures boundedness of the exact evolution operator on the underlying Banach space, it remains to express the defect $R$ in a suitably manner. 

\MyParagraph{Expansion of defect (Approach 1)}
As pointed out before, a relevant aspect in the expansion of the defect is the identification of commutators of the values of the defining operators at certain intermediate times 
\begin{equation*}
\big[A(s), B(t)\big]\,, \quad \big[A(s), C(t)\big]\,, \quad 
\big[B(s), C(t)\big]\,, \quad s, t \in [t_n, t_n + \tau_n]\,. 
\end{equation*}
Thus, we may focus on a detailed derivation of the needed auxiliary expansion rather than stating the lengthy local error expansion in full detail. 
Recalling the definition of the numerical evolution operator \begin{equation*}
\begin{gathered}  
S_F(t - t_n; t_n) \, u_n = E_C\big(t - t_n; t_n + \tfrac{2}{3} \, \tau_n; \tfrac{1}{3}\big) \, E_B\big(t - t_n; t_n + \tfrac{1}{3} \, \tau_n; \tfrac{1}{3}\big) \, E_A\big(t - t_n; t_n; \tfrac{1}{3}\big)\,, 
\quad
t \in [t_n, t_n + \tau_n]\,,
\end{gathered}  
\end{equation*}
and the related evolution equations 
\begin{equation*}
\begin{aligned}  
&\tfrac{\dd}{\dd t} \, E_A\big(t - t_n; t_n; \tfrac{1}{3}\big) = A\big(t_n + \tfrac{1}{3} \, (t - t_n)\big) \, E_A\big(t - t_n; t_n; \tfrac{1}{3}\big)\,, &&
t \in [t_n, t_n + \tau_n]\,,\\
&\tfrac{\dd}{\dd t} \, E_B\big(t - t_n; t_n + \tfrac{1}{3} \, \tau_n; \tfrac{1}{3}\big)
= B\big(t_n + \tfrac{1}{3} \, \tau_n + \tfrac{1}{3} \, (t - t_n)\big) \, E_B\big(t - t_n; t_n + \tfrac{1}{3} \, \tau_n; \tfrac{1}{3}\big)\,, &&
t \in [t_n, t_n + \tau_n]\,,\\
&\tfrac{\dd}{\dd t} \, E_C\big(t - t_n; t_n + \tfrac{2}{3} \, \tau_n; \tfrac{1}{3}\big)
= C\big(t_n + \tfrac{2}{3} \, \tau_n + \tfrac{1}{3} \, (t - t_n)\big) \, E_C\big(t - t_n; t_n + \tfrac{2}{3} \, \tau_n; \tfrac{1}{3}\big)\,, &&
t \in [t_n, t_n + \tau_n]\,,
\end{aligned}  
\end{equation*}
it is straightforward to determine its time derivative by the chain rule 
\begin{equation*}
\begin{split}  
&\tfrac{\dd}{\dd t} \, S_F(t - t_n; t_n) \\
&= \tfrac{\dd}{\dd t} \, E_C\big(t - t_n; t_n + \tfrac{2}{3} \, \tau_n; \tfrac{1}{3}\big) \, E_B\big(t - t_n; t_n + \tfrac{1}{3} \, \tau_n; \tfrac{1}{3}\big) \, E_A\big(t - t_n; t_n; \tfrac{1}{3}\big) \\
&\quad + E_C\big(t - t_n; t_n + \tfrac{2}{3} \, \tau_n; \tfrac{1}{3}\big) \, \tfrac{\dd}{\dd t} \, E_B\big(t - t_n; t_n + \tfrac{1}{3} \, \tau_n; \tfrac{1}{3}\big) \, E_A\big(t - t_n; t_n; \tfrac{1}{3}\big) \\
&\quad + E_C\big(t - t_n; t_n + \tfrac{2}{3} \, \tau_n; \tfrac{1}{3}\big) \, E_B\big(t - t_n; t_n + \tfrac{1}{3} \, \tau_n; \tfrac{1}{3}\big) \, \tfrac{\dd}{\dd t} \, E_A\big(t - t_n; t_n; \tfrac{1}{3}\big) \\
&= C\big(t_n + \tfrac{2}{3} \, \tau_n + \tfrac{1}{3} \, (t - t_n)\big) \\ 
&\quad \times E_C\big(t - t_n; t_n + \tfrac{2}{3} \, \tau_n; \tfrac{1}{3}\big) \,
E_B\big(t - t_n; t_n + \tfrac{1}{3} \, \tau_n; \tfrac{1}{3}\big) \, E_A\big(t - t_n; t_n; \tfrac{1}{3}\big) \\
&\quad + E_C\big(t - t_n; t_n + \tfrac{2}{3} \, \tau_n; \tfrac{1}{3}\big) \,
B\big(t_n + \tfrac{1}{3} \, \tau_n + \tfrac{1}{3} \, (t - t_n)\big) \\ 
&\quad \times E_B\big(t - t_n; t_n + \tfrac{1}{3} \, \tau_n; \tfrac{1}{3}\big) \, E_A\big(t - t_n; t_n; \tfrac{1}{3}\big) \\
&\quad + E_C\big(t - t_n; t_n + \tfrac{2}{3} \, \tau_n; \tfrac{1}{3}\big) \,
E_B\big(t - t_n; t_n + \tfrac{1}{3} \, \tau_n; \tfrac{1}{3}\big) \, A\big(t_n + \tfrac{1}{3} \, (t - t_n)\big) \\ &\quad \times E_A\big(t - t_n; t_n; \tfrac{1}{3}\big)\,, 
\quad t \in [t_n, t_n + \tau_n]\,.
\end{split}  
\end{equation*}
Inserting this identity into the relation for the associated defect 
\begin{equation*}
\begin{split}
&R(t - t_n; t_n) 
= \tfrac{\dd}{\dd t} \, S_F(t - t_n; t_n) 
- F(t) \, S_F(t - t_n; t_n)\,, 
\quad t \in [t_n, t_n + \tau_n]\,,
\end{split}  
\end{equation*}
we arrive at 
\begin{equation*}
\begin{split}
&R(t - t_n; t_n) \\
&= C\big(t_n + \tfrac{2}{3} \, \tau_n + \tfrac{1}{3} \, (t - t_n)\big) \\ 
&\quad \times E_C\big(t - t_n; t_n + \tfrac{2}{3} \, \tau_n; \tfrac{1}{3}\big) \,
E_B\big(t - t_n; t_n + \tfrac{1}{3} \, \tau_n; \tfrac{1}{3}\big) \, E_A\big(t - t_n; t_n; \tfrac{1}{3}\big) \\
&\quad + E_C\big(t - t_n; t_n + \tfrac{2}{3} \, \tau_n; \tfrac{1}{3}\big) \,
B\big(t_n + \tfrac{1}{3} \, \tau_n + \tfrac{1}{3} \, (t - t_n)\big) \\ 
&\quad \times E_B\big(t - t_n; t_n + \tfrac{1}{3} \, \tau_n; \tfrac{1}{3}\big) \, E_A\big(t - t_n; t_n; \tfrac{1}{3}\big) \\
&\quad + E_C\big(t - t_n; t_n + \tfrac{2}{3} \, \tau_n; \tfrac{1}{3}\big) \,
E_B\big(t - t_n; t_n + \tfrac{1}{3} \, \tau_n; \tfrac{1}{3}\big) \, A\big(t_n + \tfrac{1}{3} \, (t - t_n)\big) \\ 
&\quad - \big(A(t) + B(t) + C(t)\big) \\ 
&\quad \times E_C\big(t - t_n; t_n + \tfrac{2}{3} \, \tau_n; \tfrac{1}{3}\big) \, E_B\big(t - t_n; t_n + \tfrac{1}{3} \, \tau_n; \tfrac{1}{3}\big) \, E_A\big(t - t_n; t_n; \tfrac{1}{3}\big)\,,
\quad t \in [t_n, t_n + \tau_n]\,.
\end{split}
\end{equation*}
After suitable reformulation, this expression involves differences of the defining operators at certain intermediate times as well as commutators of defining operators and evolution operators, for instance  
\begin{equation*}
A(s) - A(t)\,, \quad \big[B(t), E_A\big(t - t_n; t_n; \tfrac{1}{3}\big)\big]\,, \quad s, t \in [t_n, t_n + \tau_n]\,.
\end{equation*}
For our purposes, it suffices to specify the arguments of $A, B, C$, but we may omit the arguments of the evolution operators. 
That is, we have 
\begin{equation*}
\begin{split}
&R(t - t_n; t_n) \\
&= C\big(t_n + \tfrac{2}{3} \, \tau_n + \tfrac{1}{3} \, (t - t_n)\big) \, E_C \, E_B \, E_A - C(t) \, E_C \, E_B \, E_A\\
&\quad + E_C \, B\big(t_n + \tfrac{1}{3} \, \tau_n + \tfrac{1}{3} \, (t - t_n)\big) \, E_B \, E_A - B(t) \, E_C \, E_B \, E_A\\
&\quad + E_C \, E_B \, A\big(t_n + \tfrac{1}{3} \, (t - t_n)\big) \, E_A - A(t) \, E_C \, E_B \, E_A \\  
&= \big(C\big(t_n + \tfrac{2}{3} \, \tau_n + \tfrac{1}{3} \, (t - t_n)\big) - C(t)\big) \, E_C \, E_B \, E_A \\
&\quad + E_C \, \big(B\big(t_n + \tfrac{1}{3} \, \tau_n + \tfrac{1}{3} \, (t - t_n)\big) - B(t)\big) \, E_B \, E_A  \\
&\quad + E_C \, E_B \, \big(A\big(t_n + \tfrac{1}{3} \, (t - t_n)\big) - A(t)\big) \, E_A \\
&\quad + \big[E_C, A(t) + B(t)\big] \, E_B \, E_A + E_C \, \big[E_B, A(t)\big] \, E_A\,, 
\quad t \in [t_n, t_n + \tau_n]\,.
\end{split}
\end{equation*}
\begin{enumerate}[(i)]
\item 
On the one hand, in view of a subsequent application to the Bloch--Torrey equation, it is reasonable to assume the temporal smoothness of the defining operators 
\begin{equation*}
\begin{gathered}
A\big(t_n + \tfrac{1}{3} \, (t - t_n)\big) - A(t) = \nO\big(- \tfrac{2}{3} \, (t - t_n)\big) = \nO(\tau_n)\,, \\
B\big(t_n + \tfrac{1}{3} \, \tau_n + \tfrac{1}{3} \, (t - t_n)\big) - B(t) = \nO\big(\tfrac{1}{3} \, \tau_n - \tfrac{2}{3} \, (t - t_n)\big) = \nO(\tau_n)\,, \\
C\big(t_n + \tfrac{2}{3} \, \tau_n + \tfrac{1}{3} \, (t - t_n)\big) - C(t) = \nO\big(\tfrac{2}{3} \, \tau_n - \tfrac{2}{3} \, (t - t_n)\big) = \nO(\tau_n)\,, \\
t \in [t_n, t_n + \tau_n]\,.
\end{gathered}
\end{equation*}
In combination with the available stability results, this implies that the first differences have the desired local orders two. 
More precisely, due to the fact that these terms satisfy
\begin{equation*}
\begin{split}
R^{(1)}(t - t_n; t_n) 
&= \big(C\big(t_n + \tfrac{2}{3} \, \tau_n + \tfrac{1}{3} \, (t - t_n)\big) - C(t)\big) \, E_C \, E_B \, E_A \\
&\quad + E_C \, \big(B\big(t_n + \tfrac{1}{3} \, \tau_n + \tfrac{1}{3} \, (t - t_n)\big) - B(t)\big) \, E_B \, E_A  \\
&\quad + E_C \, E_B \, \big(A\big(t_n + \tfrac{1}{3} \, (t - t_n)\big) - A(t)\big) \, E_A \\
&= \nO(\tau_n)\,, \quad t \in [t_n, t_n + \tau_n]\,,
\end{split}
\end{equation*}
they lead to second-order contributions in the local errors 
\begin{equation*}
L_F^{(1)}(\tau_n; t_n) 
= \int_{0}^{\tau_n} E_F(\tau_n - \zeta; t_n) \, R^{(1)}(\zeta; t_n) \; \dd \zeta = \nO\big(\tau_n^2\big)\,. 
\end{equation*}
Altogether, they yield the desired first-order contribution in the global error estimate
\begin{equation}
\label{eq:GlobalError1}
\begin{split}  
&\prod_{n=0}^{N-1} S_F(\tau_n; t_n) \, \big(u_0 - u(t_0)\big) 
+ \sum_{n=0}^{N-1} \Bigg(\prod_{m=n+1}^{N-1} S_F(\tau_m; t_m)\Bigg) \, L_F^{(1)}(\tau_n; t_n) \, \Bigg(\prod_{m=0}^{n-1} E_F(\tau_m; t_m) \Bigg) 
= \nO(\tau_n)\,,
\end{split}  
\end{equation}
see also~\eqref{eq:GlobalError}.
\item 
On the other hand, for the remaining terms, we employ differentiation of the arising commutators and an integral representation by means of the variation-of-constants formula, in order to argue that the relation 
\begin{equation*}
\begin{aligned}
R^{(2)}(t - t_n; t_n) 
&= \big[E_C, A(t) + B(t)\big] \, E_B \, E_A + E_C \, \big[E_B, A(t)\big] \, E_A 
= \nO(\tau_n)\,, \quad t \in [t_n, t_n + \tau_n]\,,     
\end{aligned}
\end{equation*}
is valid under suitable regularity assumptions.
Analogously to before, this implies 
\begin{equation*}
L_F^{(2)}(\tau_n; t_n) 
= \int_{0}^{\tau} E_F(\tau_n - \zeta; t_n) \, R^{(2)}(\zeta; t_n) \; \dd \zeta = \nO\big(\tau_n^2\big)
\end{equation*}
and yields the desired first-order contribution in the global error estimate
\begin{equation}
\label{eq:GlobalError2}
\begin{split}  
&\prod_{n=0}^{N-1} S_F(\tau_n; t_n) \, \big(u_0 - u(t_0)\big) 
+ \sum_{n=0}^{N-1} \Bigg(\prod_{m=n+1}^{N-1} S_F(\tau_m; t_m)\Bigg) \, L_F^{(2)}(\tau_n; t_n) \, \Bigg(\prod_{m=0}^{n-1} E_F(\tau_m; t_m) \Bigg) 
= \nO(\tau_n)\,.
\end{split}  
\end{equation}
Hereby, we observe that the commutators are of the general form
\begin{equation*}
f(t - t_n) = \big[E_G(t - t_n, \sigma(t_n, \tau_n), c), H(t)\big]\,, 
\quad t \in [t_n, t_n + \tau_n]\,, 
\end{equation*}
and satisfy the initial value problem 
\begin{equation*}
\begin{cases}
f'(t - t_n) = G\big(\sigma(t_n, \tau_n) + c \, (t - t_n)\big) \, f(t - t_n) + g(t - t_n)\,, \quad t \in [t_n, t_n + \tau_n]\,, \\
f(t_n) = 0\,, 
\end{cases}
\end{equation*}
since their time derivatives can be rewritten as 
\begin{equation*}
\begin{split}
f'(t - t_n) 
&= \big[G\big(\sigma(t_n, \tau_n) + c \, (t - t_n)\big) \, E_G(t - t_n, \sigma(t_n, \tau_n), c\big), H(t)\big] 
+ \big[E_G(t - t_n, \sigma(t_n, \tau_n), H'(t)\big] \\
&= G\big(\sigma(t_n, \tau_n) + c \, (t - t_n)\big) \, f(t - t_n) 
+ g(t - t_n)\,, \\
g(t - t_n) 
&= \big[G\big(\sigma(t_n, \tau_n) + c \, (t - t_n)\big), H(t)\big] \, E_G(t - t_n, \sigma(t_n, \tau_n), c\big) 
+ \big[E_G(t - t_n, \sigma(t_n, \tau_n), H'(t)\big]\,, 
\\&
\quad t \in [t_n, t_n + \tau_n]\,. 
\end{split}
\end{equation*}
A solution representation based on the variation-of-constants formula 
\begin{equation*}
f(t - t_n) = \int_{0}^{t - t_n}  E_G\big(t - t_n - \zeta; \sigma(t_n, \tau_n); c\big) \, g(\zeta) \; \dd \zeta\,, \quad t \in [t_n, t_n + \tau_n]\,,
\end{equation*}
justifies the first-order dependence on the time increment. 
\end{enumerate}

\MyParagraph{Summary (Approach 1)}
Altogether, the approach sketched above implies a second-order local error bound 
\begin{equation*}
L_F(\tau_n; t_n) = \nO\big(\tau_n^2\big)\,, \quad n \in \{1, \dots, N\}\,, 
\end{equation*}
and in consequence a first-order convergence rate for the Lie--Trotter splitting method applied to non-autonomous evolutionary problems with three parts
\begin{equation*}
u_n - u(t_n) = \nO(\tau_n)\,, \quad n \in \{1, \dots, N\}\,, 
\end{equation*}
see in particular~\eqref{eq:GlobalError1} and~\eqref{eq:GlobalError2}. 
The necessary stability assumptions and regularity requirements on the problem data are summarised in Theorem~\ref{thm:glob_err}. 
In the subsequent sections, we substantiate them by a theoretical analysis and numerical experiments for the Bloch--Torrey equation. 
That is, we specify the defining operators and justify their differentiability such that 
\begin{equation*}
\begin{gathered}
A(s) - A(t) = \nO(s - t)\,, \quad 
B(s) - B(t) = \nO(s - t)\,, \quad 
C(s) - C(t) = \nO(s - t)\,, \\
s, t \in [t_n, t_n + \tau_n]\,, \quad n \in \{1, \dots, N\}\,. 
\end{gathered}
\end{equation*}
we provide the needed stability bounds for the arising evolution operators
\begin{equation*}
E_F\,, E_A\,, E_B\,, E_C\,,
\end{equation*}
with respect to the norm of the underlying Banach space and certain subspaces, 
and deduce the decisive estimates for the commutators 
\begin{equation*}
\big[A(s), B(t)\big]\,, \quad 
\big[A(s), C(t)\big]\,, \quad 
\big[B(s), C(t)\big]\,, \quad 
s, t \in [t_n, t_n + \tau_n]\,, \quad n \in \{1, \dots, N\}\,. 
\end{equation*}

\MyParagraph{Summary (Approach 2)}
The convergence analysis of the second approach can be decomposed into two steps. 
Thereby, auxiliary results and solution representations by the variation-of-constants formula are used, which were already established for the first approach. 
On the one hand, employing the temporal differentiability of the defining operator, it is seen that the approximation error corresponding to a freezing of time fulfills 
\begin{equation*}
E_F(\tau_n; t_n) - E_{F(t_n)}(\tau_n) = \nO\big(\tau_n^2\big)\,, \quad n \in \{1, \dots, N\}\,,
\end{equation*}
see also~\MyCite{BlanesCasasThalhammer2018,GonzalezOstermannThalhammer2006} for a generalisation to higher-order Magnus-type integrators. 
On the other hand, in the simplified setting of the Lie--Trotter splitting for autonomous evolutionary problems, the arguments developed for the first approach at once yield 
\begin{equation*}
\begin{gathered}
R(t - t_n; t_n) = \tfrac{\dd}{\dd t} \, S_F(t - t_n; t_n) - F(t_n) \, S_F(t - t_n; t_n) = \nO(\tau_n)\,, \quad t \in [t_n, t_n + \tau_n]\,, \\
L_F(\tau_n; t_n) 
= \int_{0}^{\tau_n} E_F(\tau_n - \zeta; t_n) \, R(\zeta; t_n) \; \dd \zeta = \nO\big(\tau_n^2\big)\,, \quad n \in \{1, \dots, N\}\,,
\end{gathered}
\end{equation*}
see also~\MyCite{AuzingerKochThalhammer2015,BlanesCasasGonzalezThalhammer2025} for a stability and convergence analysis of high-order splitting methods. 
Altogether, by means of the provided stability results, this again ensures the desired global convergence rate 
\begin{equation*}
u_n - u(t_n) = \nO(\tau_n)\,, \quad n \in \{1, \dots, N\}\,. 
\end{equation*}

\MyParagraph{Main result}
The following result provides a first-order global error estimate for both approaches explored before, the Lie--Trotter splitting method for a non-autonomous evolutionary problem with three parts~\eqref{eq:IVP},  reformulated as autonomous problem, and the composition of an elementary  Magnus integrator based on piecewise constant interpolants with the Lie--Trotter splitting method for autonomous evolution equations comprising three operators.  

\begin{Theorem}
\label{thm:glob_err}
Consider the initial value problem for a non-autonomous evolution equation 
\begin{equation*}
\begin{cases}
u'(t) = F(t) \, u(t) = A(t) \, u(t) + B(t) \, u(t) + C(t) \, u(t)\,, \quad t \in [t_0, T]\,, \\
u(t_0) \, \text{ given}\,, 
\end{cases}
\end{equation*}
and assume that the linear operators $A(t): D_A \to X$, $B(t): D_B \to X$, and $C(t): D_C \to X$, $t \in [t_0, T]$, are defined on certain subspaces $D_A, D_B, D_C \subseteq X$ of the underlying Banach space $(X, \|\cdot\|_X)$ with non-empty intersection.
Suppose that the associated evolution operators are well-defined and in particular satisfy stability bounds of the form 
\begin{equation*}
\begin{gathered}
\|E_G(\tau;\cdot)\|_{X \leftarrow X} \leq \ee^{\, \text{Const} \, \tau}\,, \quad 
G\in\{F,\,A,\,B,\,C\}\,, \quad 
\tau \in [0, T - t_0]\,, 
\end{gathered}
\end{equation*}
with generic constants, uniformly in the second and third arguments, respectively. 
Assume further that the prescribed initial value and hence the exact solution values are bounded
\begin{equation*}
\|u(t)\|_D \leq \text{Const}\,, \quad t \in [t_0, T]\,, 
\end{equation*}
with respect to the norm of a suitably restricted subspace $D \subseteq X$ that is defined by the requirement that the commutators arising in the expansion of the local error remain bounded
\begin{equation*}
\big\|[A(s),B(t)]\big\|_{X \leftarrow D}
+ \big\|[A(s),C(t)]\big\|_{X \leftarrow D} + \big\|[B(s),C(t)]\big\|_{X \leftarrow D}\leq \text{Const}\,, \quad s, t \in [t_0, T]\,,  
\end{equation*}
and that also their compositions with evolution operators associated with the original problem and the corresponding subproblems remain bounded. 
Under these general requirements, the proposed first-order time integration method based on the Lie--Trotter splitting retains its classical order and satisfies the global error estimate 
\begin{equation*}
\big\|u_n - u(t_n)\big\|_X \leq \text{Const} \, \max_{\nu \in \{0, \dots, n-1\}} \tau_{\nu}\,, \quad
n \in \{1, \dots, N\}\,,
\end{equation*}
where the arising constant is independent of the time increment and the number of time steps. 
\end{Theorem}

\begin{Remark}
\label{rem:differentiability}
For the first approach, a sufficient abstract hypothesis for the differentiability statements above, on Banach spaces $Y \hookrightarrow X$ ($Y=X$ for pure multiplication operators), is
\begin{equation}\label{ass:timereg_abstract}
A,\, B,\, C \in W^{1,\infty}\big(0,T;\mathcal{L}(Y,X)\big)\,.
\end{equation}
Indeed, \eqref{ass:timereg_abstract} directly yields, for $s,t\in[0,T]$,
\begin{equation*}
\|A(t)-A(s)\|_{\mathcal{L}(Y,X)} \le M_A \, |t-s|\,, \quad \|A'(t)\|_{\mathcal{L}(Y,X)} \le M_A\,,
\end{equation*}
and analogously for $B, C$ (with constants $M_B, M_C$), which is precisely the differentiability and boundedness required above. For the Bloch--Torrey equation (Section~\ref{sec:BT}), Hypothesis~\eqref{ass:timereg_abstract} is verified concretely in terms of the physical coefficients $\magind, \velo$ in Proposition~\ref{prop:timereg_BT}, where in fact $B' \equiv 0$ exactly, since $B$ does not depend on time there.
For the second approach, since $F(t) = A(t) + B(t) + C(t)$, Hypothesis~\eqref{ass:timereg_abstract} gives $F' = A' + B' + C'$ bounded; for the Bloch--Torrey equation, where $B' \equiv 0$, this bound is made explicit at the end of Proposition~\ref{prop:timereg_BT}.
\end{Remark}

\section{A class of problems containing multiplication and diffusion-advection operators} \label{sec:MultDiffAdv}
We now somewhat specify the setting to operators $A$ and $B$ being multiplication operators and $C$ a component wise diffusion-advection operator. 
This is inspired by the convective Bloch-Torrey equation \eqref{eqn:bloch-torrey}, in which each of the multiplication operators models a different physical effect (namely $A$ rotation, that is subject to a time dependent control and thus inherently time dependent and $B$ relaxation), but also comprises many other relevant applications in which diffusion and/or advection as well as a time dependent bilinear control term play a role.

We consider the evolution equation on a bounded domain $\Omega\subseteq\mathbb{R}^d$, $d\in\mathbb{N}$, $d\geq3$ and at each time instance $t$, the state is a vector field $u(t):\Omega\to \mathbb{R}^d$.
For simplicity of exposition (and since it is the commonly used boundary condition in \eqref{eqn:bloch-torrey_1}), we impose homogeneous Dirichlet boundary conditions 
$u(t)=0$ on $\partial\Omega$.

The multiplication operators $A(t)$ and $B(t)$ are defined by matrix valued functions with entries $a_{ij}(t)$, $b_{ij}(t)$ $\in L^\infty(\Omega)$, which by a slight abuse of notation we again denote by $A(t)$ and $B(t)$
\begin{equation}\label{defA}
(A(t)w)_i=\sum_{j=1}^d a_{ij}(t) w_j,\quad i\in\{1,\ldots, d\}, \quad w=(w_1,\ldots,w_d)\in L^2(\Omega;\mathbb{R}^d) 
\end{equation}
and likewise for $B(t)$.

We assume the diffusion and advection to act component wise, that is, with a uniformly positive definite diffusion matrix 
\begin{equation}\label{ass:D}
D\in L^\infty(\Omega;\mathbb{R}^{d\times d}), \quad 
\xi^T\,D(x)\,\xi \geq \underline{c}_D>0 \ \text{for all } \xi\in \mathbb{R}^d, \text{ and almost all } x\in \Omega
\end{equation}
and a possibly time dependent velocity
\begin{equation}\label{ass:v}
v\in L^\infty(0,T;L^2(\Omega;\mathbb{R}^d)),
\end{equation}
we define 
\begin{equation}\label{defC}
(C(t) w )_i = \nabla\cdot (D\nabla w_i)-\velo(t)\cdot\nabla w_i, \qquad
w=(w_1,\ldots,w_d)\in H^2(\Omega;\mathbb{R}^d).
\end{equation}

\subsection{Evolution estimates}
\subsubsection{Estimate of $E_{A}$ and $E_{B}$}
Existence of a solution $E_{A}(t-t_0,t_0,q)u_0=:u_A(t)$ to the subproblem 
\begin{equation}\label{ivp_A} 
\tfrac{d}{d t} u_A(t)=A(t_0+q(t-t_0))u_A(t), \ t\in (t_0,t_0+\tau) \quad u_A(t_0)=u_0
\end{equation}
follows from standard ODE theory, applied pointwise in space. 
Here $q\in\{\frac13,\,\frac23,\,1\}$. For simplicity of exposition (and by a simple temporal rescaling argument, without loss of generality) we set $q=1$, $t_0=0$, $\tau=T$ here.

The following results additionally provides some spatial regularity.

\begin{Proposition}\label{prop:evolutionest_mult}
For any $m\in\mathbb{N}$, $A\in L^\infty(0,T;W^{1,3}(\Omega;\mathbb{R}^{d\times d})\cap H^m(\Omega;\mathbb{R}^{d\times d})$, the evolution operator $E_{A}(t-t_0,t_0,q)u_0$ is well defined by \eqref{ivp_A} and satisfies
\[
\|E_{A}(t-t_0,t_0,q)\|_{H^m(\Omega;\mathbb{R}^d)\leftarrow H^m(\Omega;\mathbb{R}^d)} \leq K_A
\]
with $K_A$ depending only on $m$, $d$, $\Omega$, $T$, and $\|A\|_{L^\infty(0,T;W^{1,3}(\Omega;\mathbb{R}^{d\times d})\cap H^m(\Omega;\mathbb{R}^{d\times d})}$.
\end{Proposition}

\subsubsection{Estimate of $E_{C}$ and of $E_{F}$}
Since the commutators lead to an at most first order derivative loss in this section and an at most second derivative loss in section~\ref{sec:BT}, we only require estimates of $E_C$ of $E_F$ in 
$L^\infty(0,T; H^s(\Omega;\mathbb{R}^d))$ for $s\in\{0,1,2\}$.
Existence and uniqueness of solutions in $L^2(0,T;H_0^1(\Omega;\mathbb{R}^d))\cap L^\infty(0,T; L^2(\Omega;\mathbb{R}^d))\cap H^1(0,T;H^{-1}(\Omega;\mathbb{R}^d))$ to the underlying parabolic initial value problems with inital data in $ L^2(\Omega;\mathbb{R}^d)$ can be obtained from standard textbook results, cf. e.g., \cite[Theorem 3, Section 7.1]{EvansBook}. We therefore only provide the relevant energy estimates.
Consider 
\begin{equation}\label{ivp_C}
\begin{aligned}
&\tfrac{d}{d t} u_C(t)=\nabla\cdot (D\nabla u_C(t))-\velo(t)\cdot\nabla u_C(t), \ t\in (t_0,t_0+\tau) \\
& u_C(t_0)=u_0
\end{aligned}
\end{equation}
and
\begin{equation}\label{ivp_F}
\begin{aligned}
&\tfrac{d}{d t} u_F(t)=(A(t)+B(t))u_F(t)+\nabla\cdot (D\nabla u_F(t))-\velo(t)\cdot\nabla u_F(t), \ t\in (t_0,t_0+\tau) \\
&u_F(t_0)=u_0, 
\end{aligned}
\end{equation}
where we assume the diffusion and advection operators to act in the component wise sense of \eqref{defC}.
Since we do not make any sign condition on the multiplication operators that could potentially improve the energy estimates, we treat the estimate of $u_C$ as a particular case of $u_F$ with $A+B\equiv0$. 

\begin{Proposition}\label{prop:evolutionest_F}
For any $m\in\{0,1\}$ and $A+B\in L^\infty(0,T;L^\infty(\Omega;\mathbb{R}^{d\times d}))$, $\velo\in L^\infty(0,T;L^\infty(\Omega))$, the evolution operator $E_{F}(t-t_0,t_0,q)u_0$ is well defined by \eqref{ivp_F} and satisfies
\[
\|E_{F}(t-t_0,t_0,q)\|_{H^m(\Omega;\mathbb{R}^d)\leftarrow H^m(\Omega;\mathbb{R}^d)} \leq K_F
\]
with $K_F$ depending only on $m$, $d$, $\Omega$, $T$, $\|A+B\|_{L^\infty(0,T;L^\infty(\Omega;\mathbb{R}^{d\times d}))}$, and $\|\velo\|_{L^\infty(0,T;L^\infty(\Omega))}$.
\\
If additionally $\Omega$ and $D\in W^{1,d}(\Omega)$ are such that the elliptic regularity estimate
\begin{equation}\label{ell-reg}
\|v\|_{H^2(\Omega)}\leq 
\overline{c}\Bigl(\|\nabla\cdot(D\nabla v)\|_{L^2(\Omega)}+\|v\|_{L^2(\Omega)}\Bigr)
\end{equation}
holds, then 
\[
\|E_{F}(t-t_0,t_0,q)\|_{H^2(\Omega;\mathbb{R}^d)\leftarrow H^2(\Omega;\mathbb{R}^d)} \leq K_F
\]
with $K_F$ depending only on $m$, $d$, $\Omega$, $T$, $\|A+B\|_{L^\infty(0,T;L^\infty(\Omega;\mathbb{R}^{d\times d}))}$, $\|\velo\|_{L^\infty(0,T;L^\infty(\Omega))}$, and $\overline{c}$.
\end{Proposition}
\textit{Proof.} See the Appendix.

\subsection{Commutator estimates}

We write down all commutators at some time instance $t$, which can be thought of being fixed to $t=t_n$ for the frozen version of the splitting method.

The multiplication commutation $[A(t),B(t)]$ can trivially be estimated by 
\[
\|[A(t),B(t)]\|_{L^2(\Omega;\mathbb{R}^d)\leftarrow L^2(\Omega;\mathbb{R}^d)}
\leq 2\|A\|_{L^\infty(0,T;L^\infty(\Omega;\mathbb{R}^{d\times d})}
\|B\|_{L^\infty(0,T;L^\infty(\Omega;\mathbb{R}^{d\times d})}
\]
and the commutator estimate $[C(t),B(t)]$ goes analogously to the one for $[C(t),A(t)]$. It thus only remains to bound the latter.

\subsubsection{Estimate of the commutator [$C(t),A(t)$]}
Since $C$ acts the same way on each component, we can write it as 
\[
\begin{aligned}
&C(t)=\mathcal{L}(t)\, I= (C_{\text{diff}}+C_{\text{adv}}(t)) I \\
&\text{ where }\ C_{\text{diff}} \phi=\nabla\cdot(D\nabla \phi), \quad C_{\text{adv}}(t)=-\velo(t)\nabla\cdot \phi, \quad \phi\in H^2(\Omega)
\end{aligned}
\]
For a general multiplication operator $A(t)$ defined by \eqref{defA} with functions $a_{ij}(t)\in L^\infty(\Omega)$ (further regularity requirements will be made in due course), we consider the diffusion and advection commutators separately, using the identity 
\[
[C(t),A(t)]=[\mathcal{L}(t) I,A(t)]=
[C_{\text{diff}} I,A(t)]+[C_{\text{adv}} I,A(t)].
\]

\begin{Proposition}\label{prop:commutatorestCA}
For 
\begin{equation}\label{ass:ADV_comm}
\begin{aligned}
&A\in L^\infty(0,T;W^{1,\infty}(\Omega;\mathbb{R}^{d\times d})\cap W^{2,d}(\Omega;\mathbb{R}^{d\times d})), \\
&D\in L^\infty(\Omega;\mathbb{R}^{d\times d})\cap W^{1,d}(\Omega;\mathbb{R}^{d\times d}), \\
&\velo\in L^\infty(0,T;L^{2d/(d-2)}(\Omega;\mathbb{R}^d)),
\end{aligned}
\end{equation}
and $C$ as defined in \eqref{defC}, the commutator $[C(t),A(t)]$ satisfies
\[
\|[C(t),A(t)]\|_{L^2(\Omega;\mathbb{R}^d)\leftarrow H^1(\Omega;\mathbb{R}^d)} \leq K_{CA}
\]
with $K_{CA}$ depending only on $m$, $d$, $\Omega$, $T$, and the norms of $A$, $D$, $v$ appearing in \eqref{ass:ADV_comm}.
\end{Proposition}
\textit{Proof.} See the Appendix.

Applying Theorem~\ref{thm:glob_err} with $X=L^2(\Omega;\mathbb{R}^d)$, $D= H^1(\Omega;\mathbb{R}^d)$ we obtain the following global error estimate.
\begin{Corollary}
For $A$, $B$, $C$ defined by \eqref{defA},
\eqref{defC}, under conditions \eqref{ass:D}, \eqref{ass:ADV_comm} (and analogously for $B$), together with the temporal regularity \eqref{ass:timereg_abstract}, the global error estimate
\begin{equation*}
\|u_n - u(t_n)\|_{L^2(\Omega)^d} \leq  C \max_{\nu \in \{0, \dots, n-1\}} \tau_{\nu}\,, \quad
n \in \{1, \dots, N\} 
\end{equation*}
holds with a constant $C>0$ independent of $N$ and $\tau_{\nu}$, $\nu \in \{0, \dots, N-1\}$. 
\end{Corollary}

\section{Application to the Bloch-Torrey equation}\label{sec:BT}
We focus on Decomposition I, where 
\begin{equation}\label{ABC_BlochTorrey}
\begin{aligned}
&A(t)v:=\gamma\begin{pmatrix}\magind_z(t)v_\perp^\vdash-v_z \magind_\perp(t)^\vdash\\ \magind_\perp(t)^\vdash\cdot v_\perp\\ 0\end{pmatrix},\quad 
Bv:=\begin{pmatrix}-R_2v_\perp\\ -R_1(v_z-v_e)\\ 0\end{pmatrix},\quad
C(t)v:= \text{diag}(\mathcal{L}(t),\,\mathcal{L}(t),\,\mathcal{L}(t),\,0), \\ 
&\mathcal{L}(t)=C_{\text{diff}}+C_{\text{adv}}(t)
\end{aligned}
\end{equation}
and likewise for the frozen version with $A_n=A(t_n)$; note that $B$ is independent of time. 

While the evolution estimates can simply be taken from the previous section (and are trivial for the $v_e$ component), the commutators differ, in particular because of the skew-symmetry of  $A(t)$ and the vanishing last component of $C(t)$.

\subsubsection{Commutators for Bloch-Torrey}
For the commutator $[B,A(t)]=B(A(t)v)-A(t)(Bv)$, we obtain 
\begin{equation}
[B,A(t)]v
=\gamma\begin{pmatrix}\big((R_2-R_1)v_z+R_1v_e\big) \magind_\perp(t)^\vdash\\ (R_2-R_1)\magind_\perp(t)^\vdash\cdot v_\perp\\ 0\end{pmatrix}.
\end{equation}
Interestingly, the $\magind_z(t)$ 
term cancels out in the commutator, so that no regularity assumptions need to be imposed on $\magind_z(t)$ for estimating the commutator.

The commutator $[C(t),A(t)]=C(t)(A(t)v)-A(t)((C(t)v)$ follows the rules of the previous section
\begin{equation}
[C(t),A(t)]v
=\gamma\begin{pmatrix}
-[\mathcal{L}(t),\magind_\perp(t)^\vdash]v_z+[\mathcal{L}(t),\magind_z(t)]v_\perp^\vdash\\ 
-[\mathcal{L}(t),\magind_\perp(t)^\vdash\cdot]v_\perp\\
0
\end{pmatrix},
\end{equation}
and thus only leads to a one derivative loss in $v_\perp$, $v_z$, while leaving $v_e$ untouched.

However for the commutator $[C(t),B]=C(t)(Bv)-B((C(t)v)$, the appearance of $v_e$ in the longitudinal component of $B$ and the vanishing last component of both $B$ and $C(t)$ make a difference. We obtain
\begin{equation}
[C(t),B]v
=\begin{pmatrix}
-[\mathcal{L}(t),R_2]v_\perp\\ 
-[\mathcal{L}(t),R_1]v_z+\mathcal{L}(t)(R_1v_e)\\
0
\end{pmatrix};
\end{equation}
as a consequence, application of $[C(t),B]$ reduces spatial regularity by two orders in the last component $v_e$ and by one order in the first and second component $v_\perp$, $v_z$.

\begin{Proposition}\label{prop:commutatorest_BT}
\begin{itemize}
\item
For 
\begin{equation}\label{ass:comm_BA_BT}
\magind_\perp \in L^\infty(0,T;L^\infty(\Omega;\mathbb{R}^2))
\end{equation}
the commutator $[B,A(t)]$ satisfies
\[
\|[B,A(t)]\|_{L^2(\Omega;\mathbb{R}^4)\leftarrow L^2(\Omega;\mathbb{R}^4)} \leq K_{BA}
\]
\item
For 
\begin{equation}\label{ass:comm_CA_BT}
\begin{aligned}
&\magind_\perp \in L^\infty(0,T;W^{1,\infty}(\Omega;\mathbb{R}^2)\cap W^{2,3}(\Omega;\mathbb{R}^2)),\quad 
\magind_z \in L^\infty(0,T;W^{1,\infty}(\Omega;\mathbb{R})\cap W^{2,3}(\Omega;\mathbb{R})),\\
&D\in L^\infty(\Omega;\mathbb{R}^{d\times d})\cap W^{1,3}(\Omega;\mathbb{R}^{3\times 3}), \\
&\velo\in L^\infty(0,T;L^6(\Omega;\mathbb{R}^3)),
\end{aligned}
\end{equation}
the commutator $[C(t),A(t)]$ satisfies
\[
\|[C(t),A(t)]\|_{L^2(\Omega;\mathbb{R}^4)\leftarrow H^1(\Omega;\mathbb{R}^3)\times L^2(\Omega;\mathbb{R})} \leq K_{CA}
\]
\item
For 
\begin{equation}\label{ass:comm_CB_BT}
\begin{aligned}
&R_1,\, R_2\in W^{1,\infty}(\Omega;\mathbb{R})\cap W^{2,3}(\Omega;\mathbb{R})
\\
&D\in L^\infty(\Omega;\mathbb{R}^{d\times d})\cap W^{1,3}(\Omega;\mathbb{R}^{3\times 3}), \\
&\velo\in L^\infty(0,T;L^6(\Omega;\mathbb{R}^3)),
\end{aligned}
\end{equation}
the commutator $[C(t),B]$ satisfies
\[
\|[C(t),B]\|_{L^2(\Omega;\mathbb{R}^4)\leftarrow H^1(\Omega;\mathbb{R}^3)\times H^2(\Omega;\mathbb{R})} \leq K_{CB}
\]
\end{itemize}
Here, $K_{BA}$, $K_{CA}$, $K_{CB}$ only depend on $\Omega$, $T$, and the norms appearing in \eqref{ass:comm_BA_BT}, \eqref{ass:comm_CA_BT}, \eqref{ass:comm_CB_BT}.
\end{Proposition}
\textit{Proof.} See the Appendix.

\subsubsection{Temporal regularity for Bloch-Torrey}
The abstract hypotheses of Section~\ref{sec:GeneralFramework} also require boundedness of $A'$, $B'$, $C'$ and $O(\tau_n)$-differences of the frozen and exact operators. These translate into the following concrete conditions on $\magind$ and $\velo$.
\begin{Proposition}\label{prop:timereg_BT}
For
\begin{equation}\label{ass:timereg_BT}
\magind\in W^{1,\infty}(0,T;L^\infty(\Omega;\mathbb{R}^3)),\qquad \velo\in W^{1,\infty}(0,T;L^\infty(\Omega;\mathbb{R}^3)),
\end{equation}
there is a constant $C_A$, depending only on $\Omega$, such that
\begin{align*}
\|A(t)-A(s)\|_{L^2(\Omega;\mathbb{R}^4)\leftarrow L^2(\Omega;\mathbb{R}^4)} &\le C_A\gamma\,|t-s|\,\|\pat\magind\|_{L^\infty(0,T;L^\infty(\Omega))}\,,\\
\|A'(t)\|_{L^2(\Omega;\mathbb{R}^4)\leftarrow L^2(\Omega;\mathbb{R}^4)} &\le C_A\gamma\,\|\pat\magind(t)\|_{L^\infty(\Omega)}\,,\\
\|C(t)-C(s)\|_{L^2(\Omega;\mathbb{R}^4)\leftarrow H^1(\Omega;\mathbb{R}^3)\times L^2(\Omega)} &\le |t-s|\,\|\pat\velo\|_{L^\infty(0,T;L^\infty(\Omega))}\,,\\
\|C'(t)\|_{L^2(\Omega;\mathbb{R}^4)\leftarrow H^1(\Omega;\mathbb{R}^3)\times L^2(\Omega)} &\le \|\pat\velo(t)\|_{L^\infty(\Omega)}\,.
\end{align*}
Here the stationary diffusion part $C_{\text{diff}}$ cancels exactly in $C(t)-C(s)$ and in $C'(t)$; boundedness of the full operator $C(t)$ from $H^1(\Omega;\mathbb{R}^3)\times L^2(\Omega)$ to $L^2(\Omega;\mathbb{R}^4)$ is \emph{not} claimed here.
Since $R_1,R_2$ depend only on $x$, $B$ is autonomous: $B(t)-B(s)=0$ and $B'\equiv0$ exactly, with no regularity assumption needed.

Consequently, for $t\in[t_n,t_n+\tau_n]$, the differences arising in Approach~1 satisfy
\[
\Big\|A\big(t_n+\tfrac13(t-t_n)\big)-A(t)\Big\|_{L^2(\Omega;\mathbb{R}^4)\leftarrow L^2(\Omega;\mathbb{R}^4)} \le \tfrac23\,C_A\gamma\,\tau_n\,\|\pat\magind\|_{L^\infty_tL^\infty_x},
\]
\[
\Big\|C\big(t_n+\tfrac23\tau_n+\tfrac13(t-t_n)\big)-C(t)\Big\|_{L^2(\Omega;\mathbb{R}^4)\leftarrow H^1(\Omega;\mathbb{R}^3)\times L^2(\Omega)} \le \tfrac23\,\tau_n\,\|\pat\velo\|_{L^\infty_tL^\infty_x},
\]
while the corresponding difference for $B$ vanishes identically. For Approach~2, since $B'=0$, $F'(t)=A'(t)+C'(t)$ satisfies
\[
\|F'(t)\|_{L^2(\Omega;\mathbb{R}^4)\leftarrow H^1(\Omega;\mathbb{R}^3)\times L^2(\Omega)} \le C_A\gamma\,\|\pat\magind(t)\|_{L^\infty(\Omega)} + \|\pat\velo(t)\|_{L^\infty(\Omega)}.
\]
\end{Proposition}
\textit{Proof.} See the Appendix.

Theorem~\ref{thm:glob_err} with $X=L^2(\Omega;\mathbb{R}^4)$, $D= H^1(\Omega;\mathbb{R}^3)\times L^2(\Omega;\mathbb{R})$ yields the following global error estimate.
\begin{Corollary}
For $A$, $B$, $C$ defined by \eqref{ABC_BlochTorrey}, under the assumptions \eqref{ass:comm_CA_BT}, \eqref{ass:comm_CB_BT} (which subsume \eqref{ass:comm_BA_BT}) and the temporal regularity \eqref{ass:timereg_BT}, the global error estimate
\begin{equation*}
\|u_n - u(t_n)\|_{L^2(\Omega)^d} \leq  C \max_{\nu \in \{0, \dots, n-1\}} \tau_{\nu}\,, \quad
n \in \{1, \dots, N\} 
\end{equation*}
holds with a constant $C>0$ independent of $N$ and $\tau_{\nu}$, $\nu \in \{0, \dots, N-1\}$. 
\end{Corollary}

\def\mag{\vec{M}}
\section{Numerical implementation and tests}\label{sec:Numerics}

\subsection{Spatial discretization}

To arrive at a full discretization of the Bloch-Torrey equation, we employ a conforming $P_1$ finite element approximation.
The resulting spatial discretization error, as well as the total error between the exact PDE solution and the numerical
approximation obtained by applying the Lie splitting scheme to the spatially discretized problem, are investigated
in \cite{ReconstructionWriteupBKPM}. The rates obtained there are expected to match the standard parabolic finite
element theory for $P_1$ elements (second order in $L^2$ by an Aubin--Nitsche duality argument, first order in $H^1$
directly from the interpolation estimate); the contribution of \cite{ReconstructionWriteupBKPM} is to make this
precise for the full coupled rotation--relaxation--diffusion--advection system.
Precisely, it is shown there that, provided $\Omega$ is a sufficiently regular domain so that \eqref{ell-reg} holds and for coefficients satisfying \eqref{ass:D},
\begin{equation*}
\velo\in L^\infty\bigl(0,T;W^{1,\infty}(\Omega)^3\bigr), \quad
R_1, \, R_2\in L^\infty\bigl(0,T;L^\infty(\Omega)\bigr),
\end{equation*}
and with an initial approximation on a mesh of size $h$ satisfying
\begin{equation*}
\|\mag(0)-\mag_h(0)\|_{L^2(\Omega)^3}
\leq
K h^2,
\end{equation*}
the finite element approximation $\mag_h$ (defined below) satisfies an error estimate
\begin{equation*}
\|\mag-\mag_h\|_{L^\infty(0,T;L^2(\Omega)^3)}
\leq
C_{\mathrm{sp}}h^2
\end{equation*}
and
\begin{equation*}
\|\mag-\mag_h\|_{L^2(0,T;H^1(\Omega)^3)}
\leq
C_{\mathrm{sp}}h
\end{equation*}
holds, provided the exact solution to the Bloch-Torrey equation satisfies
\begin{equation*}
\mag\in L^\infty\bigl(0,T;H^2(\Omega)^3\bigr),
\qquad
\partial_t \mag\in L^2\bigl(0,T;H^2(\Omega)^3\bigr).
\end{equation*}
Let $\{\mathcal{T}_h\}_{h>0}$ be a family of shape-regular triangulations of $\Omega$, with mesh parameter
\begin{equation*}
    h:=\max_{K\in\mathcal{T}_h}\operatorname{diam}(K).
\end{equation*}
We abbreviate $V=H_0^1(\Omega)$ according to our setting with homogeneous Dirichlet boundary conditions; alternatively, homogeneous Neumann or no-flux boundary conditions could be taken into account by setting $V=H^1(\Omega)$.

For each triangulation $\mathcal{T}_h$, let $V_h\subset V$ be the corresponding conforming first order Lagrange finite element space, and recall the standard interpolation and Ritz-projection estimates for such spaces \cite{BrennerScott2008,Ciarlet2002,ErnGuermond2004}.
We define the continuous and discrete magnetization spaces by
\begin{equation*}
    \mathbf{V}^{\mathrm{mag}}:=[V]^3,
    \qquad
    \mathbf{V}_h^{\mathrm{mag}}:=[V_h]^3
\end{equation*}
and denote the physical magnetization and its semidiscrete finite element approximation by $\mag\in\mathbf{V}^{\mathrm{mag}}$ and 
$\mag_h\in\mathbf{V}_h^{\mathrm{mag}}$, respectively.
The latter is defined as follows.

For $w,\varphi\in\mathbf{V}$, define the diffusion bilinear form by
\begin{equation*}
    a_D(w,\varphi)
    :=
    \sum_{j=1}^{3}
    \int_\Omega
    D(\vec{r})\nabla w_j
    \cdot
    \nabla\varphi_j
    \,d\vec{r}.
\end{equation*}
and the rotation, relaxation and advection bilinear forms by
\begin{equation}\label{b_rot_rel_adv}
\begin{aligned}
    &b_{\mathrm{rot}}(t;w,\varphi):=
    \gamma\int_\Omega
    \bigl(\magind(t,\vec{r})\times w(\vec{r})\bigr)
    \cdot\varphi(\vec{r})
    \,d\vec{r},\\
    &b_{\mathrm{rel}}(w,\varphi):=
    \int_\Omega
    \Bigl(R_2(\vec{r})w_x(\vec{r})\varphi_x(\vec{r})+R_2(\vec{r})w_y(\vec{r})\varphi_y(\vec{r})
    +R_1(\vec{r})w_z(\vec{r})\varphi_z(\vec{r})\Bigr)
    \,d\vec{r},\\
    &b_{\mathrm{adv}}(t;w,\varphi):=
    \int_\Omega
    \bigl(\vec{v}(t,\vec{r})\cdot\nabla w(\vec{r})\bigr)
    \cdot\varphi(\vec{r})
    \,d\vec{r},
\end{aligned}
\end{equation}
where $\bigl(\vec{v}\cdot\nabla w\bigr)\cdot\varphi:=
    \sum_{j=1}^3\bigl(\vec{v}\cdot\nabla w_j\bigr)\varphi_j$.
The relaxation term $b_{\mathrm{rel}}$ drives the transversal components $w_x,w_y$
towards zero and the longitudinal component $w_z$ towards the equilibrium
magnetization $M^{\mathrm{eq}}$; the latter enters as a source term
\begin{equation*}
    \ell(\varphi):=\int_\Omega R_1(\vec{r})\,M^{\mathrm{eq}}(\vec{r})\,\varphi_z(\vec{r})\,d\vec{r},
\end{equation*}
independent of $t$ since $R_1$ and $M^{\mathrm{eq}}$ are stationary.

The continuous weak formulation of the Bloch-Torrey equation is
\begin{equation*}
    \bigl(\partial_t \mag,\varphi\bigr)_{L^2(\Omega)^3}
    +
    a_D(\mag,\varphi)
    +
    b(t;\mag,\varphi)
    =
    \ell(\varphi),
    \qquad
    \forall\varphi\in\mathbf{V},
\end{equation*}
where
\begin{equation*}
    b(t;w,\varphi)
    =
    b_{\mathrm{adv}}(t;w,\varphi)
    +
    b_{\mathrm{rel}}(w,\varphi)
    +
    b_{\mathrm{rot}}(t;w,\varphi).
\end{equation*}
The semidiscrete finite element problem consists in finding
$\mag_h(t)\in\mathbf{V}_h^{\mathrm{mag}}$ such that
\begin{equation}\label{Galerkin}
    \bigl(\partial_t \mag_h,\varphi_h\bigr)_{L^2(\Omega)^3}
    +
    a_D(\mag_h,\varphi_h)
    +
    b(t;\mag_h,\varphi_h)
    =
    \ell_h(\varphi_h),
    \qquad
    \forall\varphi_h\in\mathbf{V}_h^{\mathrm{mag}},
\end{equation}
where $\ell_h(\varphi_h):=\int_\Omega R_1(\vec{r})\,M_{e,h}(\vec{r})\,\varphi_{h,z}(\vec{r})\,d\vec{r}$
uses the discrete approximation $M_{e,h}\in W_h$ of $M^{\mathrm{eq}}$ introduced next.

The auxiliary variable $M_e$ is approximated in a finite-dimensional space
$W_h\subset L^2(\Omega)$.
Note that $W_h=V_h$ is not necessarily required. Since $M^{\mathrm{eq}}$ (and its discretization $M_{e,h}\in W_h$) is independent of time, the source term $\ell_h$ only needs to be well defined variationally, and
\[
\bigl(R_1M_{e,h},\varphi_{z,h}\bigr)_{L^2(\Omega)}
\]
is meaningful for every $\varphi_{z,h}\in V_h$ as soon as $R_1\in L^\infty(\Omega)$, regardless of whether $W_h=V_h$: algebraically this term is represented by a (possibly rectangular) mixed mass matrix between $W_h$ and $V_h$, and no projection of $M_{e,h}$ onto $V_h$ is needed. See \cite{ReconstructionWriteupBKPM} for further comments on this.

\subsection{Numerical experiments}

On the cylindrical domain $\Omega=\{(x,y)\, : \, x^2+y^2\leq R\}\times (0,H)$, we use
manufactured fields of the form
\begin{equation}
    \mag^{\star}(t,\vec{r}) =
        e^{-\alpha t}
        \begin{pmatrix} a_1 \\ a_2 \\ a_3 \end{pmatrix}
        \psi(\vec{r}),
    \qquad
    \psi(x,y,z) = \bigl(R^2 - x^2 - y^2\bigr)\, x \sin\!\left(\tfrac{\pi z}{H}\right),
\end{equation}
common to all cases, to estimate the discretization error of the scheme: substituting
$\mag^\star$ into the discrete Bloch--Torrey equations produces a residual
$\mathbf{R}(t,\vec{r})$, which is added to the solver as a forcing term, so that the
rate at which the numerical solution $\mag_h$ converges to $\mag^{\star}$ as $h \to 0$
(spatial study) or $\Delta t \to 0$ (temporal study) verifies the implementation.
Here $\alpha = 1.0$, $a_1 = 1.0$, $a_2 = 0.5$, $a_3 = 0.3$; the spatial factor $\psi$
vanishes on the lateral boundary ($x^2+y^2=R^2$) and on the top and bottom caps
($z=0$, $z=H$), so $\mag^{\star}$ satisfies homogeneous Dirichlet boundary conditions
on all faces. The physical parameters used are $T_1 = 1.0\,\text{s}$, $T_2 = 0.1\,\text{s}$,
$M_{\text{eq}} = 1.0$, $D_0 = 9\times 10^{-3}\,\text{m}^2/\text{s}$,
and $V_{\max} = 0.05\,\text{m/s}$ for the cases involving flow.

For all three scenarios below, the diffusion and velocity are set to the most
demanding transport configuration considered in this work: heterogeneous diffusion
\[
D(r,z) = D_0\bigl(1+\alpha_D(1-r^2/R^2)\bigr)\bigl(1+\eta_D\, z/H\bigr), \qquad
\alpha_D=0.5,\ \eta_D=0.3,
\]
combined with steady Poiseuille flow $v_z(r) = V_{\max}(1-r^2/R^2)$ (the simpler
configurations of uniform diffusion and/or vanishing flow were also verified, with
the same convergence order). Three relaxation profiles $T_1(\vec r), T_2(\vec r),
M^{eq}(\vec r)$, taken from published tissue models, are tested against this
transport configuration:
\begin{itemize}
    \item \textbf{Case~B} (surface relaxation, Brownstein--Tarr model~\cite{brownstein1979importance}):
          $T_1=T_1^{\mathrm{ref}}$ constant, $M^{eq}=M_0^{eq}$ constant, and
          \[
          T_2(r) = T_2^{w} + (T_2^{a}-T_2^{w})\Bigl(1-\frac r R\Bigr)^{\alpha_{SR}},
          \]
          with $T_1^{\mathrm{ref}}=1.0$~s, $T_2^{a}=0.150$~s, $T_2^{w}=0.040$~s, $\alpha_{SR}=0.5$,
          $M_0^{eq}=1.0$;
    \item \textbf{Case~C} (blood-oxygenation/BOLD effect~\cite{ogawa1990brain,yablonskiy1994theory}),
          with $s_{O_2}(z)=s_0-\Delta s\, z/H$:
          \begin{align*}
          T_1(r) &= T_1^{w} + (T_1^{a}-T_1^{w})\Bigl(1-\frac r R\Bigr), &
          T_2(z) &= \frac{T_2^{\max}}{1+C_{\rm oxy}(1-s_{O_2}(z))^2}, \\
          M^{eq}(r) &= M_a - (M_a-M_w)\Bigl(\frac r R\Bigr)^2, &&
          \end{align*}
          with $T_1^{a}=1.200$~s, $T_1^{w}=0.900$~s,
          $T_2^{\max}=0.150$~s, $C_{\rm oxy}=4.0$, $s_0=0.98$, $\Delta s=0.68$, $M_a=1.0$, $M_w=0.5$;
    \item \textbf{Case~D} (two-compartment model~\cite{stanisz1997analytic}), with sigmoidal interface
          $f(r)=\tfrac12\bigl(1+\tanh\bigl(\tfrac{r-r_0}{\delta r}\bigr)\bigr)$, $r_0=0.5R$, $\delta r=0.15R$:
          \begin{align*}
          T_1(r,z) &= \bigl(T_1^{c}+\Delta T_1\, f(r)\bigr)\Bigl(1+\eta_{T_1}\bigl(\tfrac z H-\tfrac12\bigr)\Bigr), &
          T_2(r) &= T_2^{c}+\Delta T_2\, f(r), \\
          M^{eq}(r,z) &= \bigl(M^{c}+\Delta M\, f(r)\bigr)\Bigl(1-\eta_M\,\tfrac z H\Bigr), &&
          \end{align*}
          with $T_1^{c}=0.80$~s, $T_1^{w}=1.50$~s, $\eta_{T_1}=0.25$, $T_2^{c}=0.050$~s,
          $T_2^{w}=0.150$~s, $M^{c}=0.65$, $M^{w}=1.00$, $\eta_M=0.20$.
\end{itemize}
The finest mesh ($33{,}449$ vertices) is fixed and the time step is refined over ten
levels, from $\Delta t = 0.25\,\text{s}$ down to $\Delta t \approx 5\times 10^{-4}\,\text{s}$.
Table~\ref{tab:temporal_mms_letter} reports
$e_h = \|\mag_h(T_{\mathrm{end}})-\mag^\star(T_{\mathrm{end}})\|_{L^2(\Omega)}$ and the
empirical rate between consecutive levels, for the frozen splitting.

\begin{table}[htbp]
\centering
\caption{Temporal convergence of the $L^2(\Omega)$ error at $t = T_{\mathrm{end}}$
         for Cases~B, C, D (fine mesh, $n_v = 33{,}449$, frozen splitting).}
\label{tab:temporal_mms_letter}
\begin{tabular}{rr rr rr rr}
\hline
$\Delta t$ (s) & $N_t$
    & $e_h$ (B) & rate
    & $e_h$ (C) & rate
    & $e_h$ (D) & rate \\
\hline
2.50e-1 &    4 & 4.744e-1 & ---  & 4.414e-1 & ---  & 3.982e-1 & ---  \\
1.25e-1 &    8 & 2.143e-1 & 1.15 & 1.879e-1 & 1.23 & 1.699e-1 & 1.23 \\
6.25e-2 &   16 & 1.023e-1 & 1.07 & 8.658e-2 & 1.12 & 7.829e-2 & 1.12 \\
3.13e-2 &   32 & 5.006e-2 & 1.03 & 4.160e-2 & 1.06 & 3.761e-2 & 1.06 \\
1.56e-2 &   64 & 2.477e-2 & 1.02 & 2.040e-2 & 1.03 & 1.844e-2 & 1.03 \\
7.81e-3 &  128 & 1.232e-2 & 1.01 & 1.010e-2 & 1.01 & 9.126e-3 & 1.01 \\
3.91e-3 &  256 & 6.141e-3 & 1.00 & 5.023e-3 & 1.01 & 4.540e-3 & 1.01 \\
1.95e-3 &  512 & 3.067e-3 & 1.00 & 2.507e-3 & 1.00 & 2.268e-3 & 1.00 \\
9.77e-4 & 1024 & 1.539e-3 & 1.00 & 1.260e-3 & 0.99 & 1.144e-3 & 0.99 \\
5.00e-4 & 2000 & 8.042e-4 & 0.97 & 6.655e-4 & 0.95 & 6.121e-4 & 0.93 \\
\hline
\end{tabular}
\end{table}

\begin{figure}[htbp]
    \centering
    \includegraphics[width=0.55\linewidth]{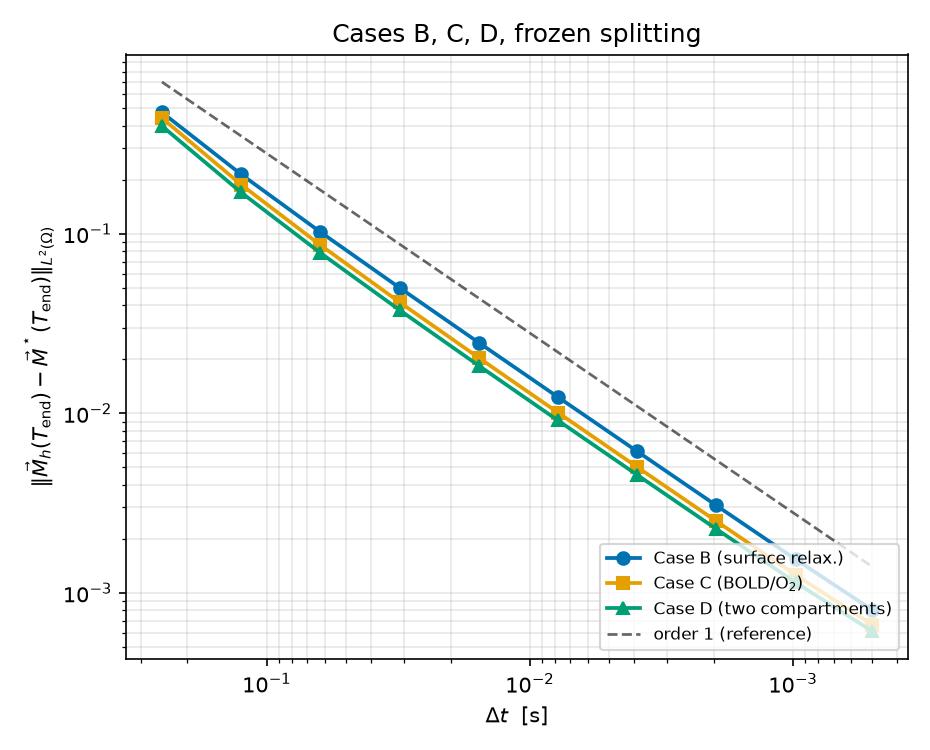}
    \caption{Temporal convergence study (MMS, frozen splitting), Cases~B, C, D:
    $L^2(\Omega)$ error at $t=T_{\mathrm{end}}$ versus time step $\Delta t$.}
    \label{fig:temporal_convergence}
\end{figure}

Table~\ref{tab:temporal_mms_letter_nf} and Fig.~\ref{fig:temporal_convergence_nf}
repeat this study for the non-frozen splitting (Approach~2 of
Section~\ref{sec:GeneralFramework}).

\begin{table}[htbp]
\centering
\caption{Temporal convergence of the $L^2(\Omega)$ error at $t = T_{\mathrm{end}}$
         for Cases~B, C, D (fine mesh, $n_v = 33{,}449$, non-frozen splitting).}
\label{tab:temporal_mms_letter_nf}
\begin{tabular}{rr rr rr rr}
\hline
$\Delta t$ (s) & $N_t$
    & $e_h$ (B) & rate
    & $e_h$ (C) & rate
    & $e_h$ (D) & rate \\
\hline
2.50e-1 &    4 & 4.237e-1 & ---  & 3.714e-1 & ---  & 3.336e-1 & ---  \\
1.25e-1 &    8 & 2.007e-1 & 1.08 & 1.689e-1 & 1.14 & 1.522e-1 & 1.13 \\
6.25e-2 &   16 & 9.760e-2 & 1.04 & 8.006e-2 & 1.08 & 7.213e-2 & 1.08 \\
3.13e-2 &   32 & 4.812e-2 & 1.02 & 3.895e-2 & 1.04 & 3.507e-2 & 1.04 \\
1.56e-2 &   64 & 2.390e-2 & 1.01 & 1.920e-2 & 1.02 & 1.729e-2 & 1.02 \\
7.81e-3 &  128 & 1.190e-2 & 1.01 & 9.534e-3 & 1.01 & 8.584e-3 & 1.01 \\
3.91e-3 &  256 & 5.941e-3 & 1.00 & 4.750e-3 & 1.01 & 4.277e-3 & 1.01 \\
1.95e-3 &  512 & 2.969e-3 & 1.00 & 2.373e-3 & 1.00 & 2.139e-3 & 1.00 \\
9.77e-4 & 1024 & 1.491e-3 & 0.99 & 1.195e-3 & 0.99 & 1.082e-3 & 0.98 \\
5.00e-4 & 2000 & 7.815e-4 & 0.97 & 6.345e-4 & 0.95 & 5.827e-4 & 0.93 \\
\hline
\end{tabular}
\end{table}

\begin{figure}[htbp]
    \centering
    \includegraphics[width=0.55\linewidth]{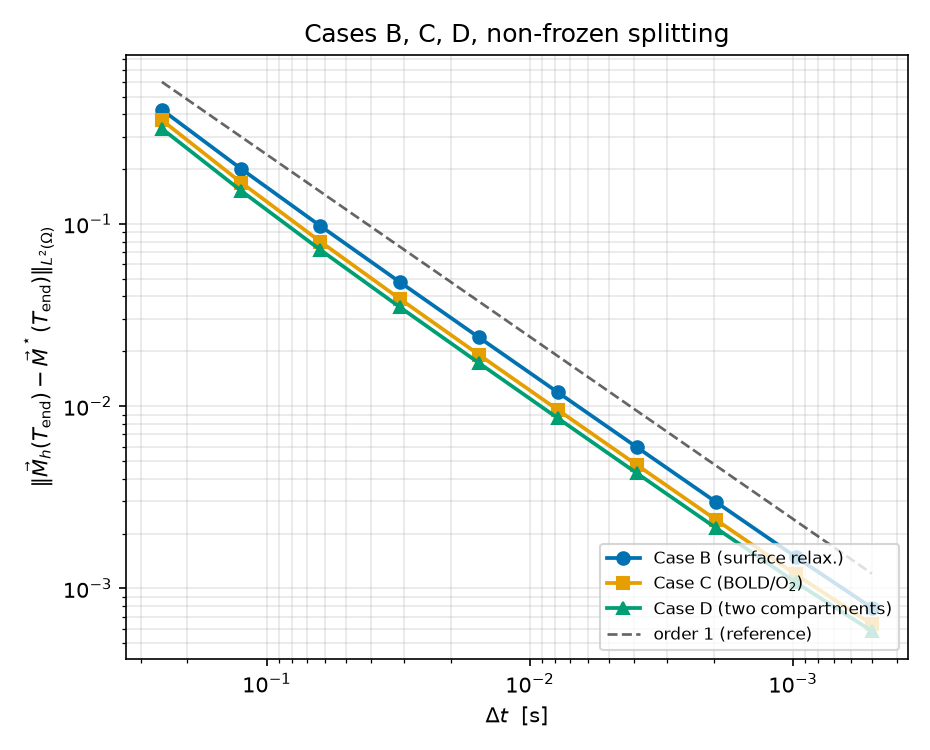}
    \caption{Temporal convergence study (MMS, non-frozen splitting), Cases~B, C, D:
    $L^2(\Omega)$ error at $t=T_{\mathrm{end}}$ versus time step $\Delta t$.}
    \label{fig:temporal_convergence_nf}
\end{figure}

\noindent
All cases exhibit $\mathcal{O}(\Delta t)$ convergence, consistent with the
first-order Lie splitting used for the temporal discretization, irrespective of the
frozen or non-frozen treatment of the non-autonomous sub-steps, confirming the
theoretical findings.

\clearpage

\section*{Appendix}

\subsection*{Proof of Proposition~\ref{prop:evolutionest_mult}}

To obtain spatial regularity of $u_A(t)$, we apply a vanishing viscosity approach, approximating $u_A$ by the solution $u_A^\varepsilon$ of a generalized heat equation 
\[
\tfrac{d}{d t} u_A^\varepsilon(t)+\varepsilon(-\Delta_D)^m u_A^\varepsilon(t)
=A(t)u_A^\varepsilon(t), \ 
t\in (0,T) \quad u_A^\varepsilon(0)=u_0
\]
for small $\varepsilon>0$; here $m\in\mathbb{N}$ quantifies the desired regularity and $-\Delta_D$ denotes the standard Dirichlet Laplacian.
By a simple extension of standard textbook arguments (cf., e.g., \cite[Chapter 7]{EvansBook}) for each $\varepsilon>0$ we have existence of a solution 
\[
u_A^\varepsilon\in H^1(0,T;L^2(\Omega;\mathbb{R}^d))\cap 
L^2(0,T;H^{2m}(\Omega;\mathbb{R}^d))\cap L^\infty(0,T;H^m(\Omega;\mathbb{R}^d))
\]
and the above generalized heat equation holds in an $L^2(0,T;L^2(\Omega;\mathbb{R}^d))$ sense, provided 
\[
A\in L^\infty(0,T;W^{1,3}(\Omega;\mathbb{R}^{d\times d})\cap H^m(\Omega;\mathbb{R}^{d\times d})\,.
\]
An $\varepsilon$ independent energy bound can be obtained by testing (that is, multiplying and integrating with respect to space and time) with 
$(-\Delta_D)^m u_A^\varepsilon$ (which is indeed an element of $L^2(0,T;L^2(\Omega;\mathbb{R}^d))$)
and integrating by parts, using the homogeneous boundary conditions 
\[
(-\Delta_D)^k u_A^\varepsilon=0\text{ on }\partial\Omega, \quad k\in\{0,\ldots,m-1\}
\]
as follows
\[
\begin{aligned}
&\|\nabla^m u_{A,i}^\varepsilon(t)\|_{L^2(\Omega)}^2
-\|\nabla^m u_{0,i}\|_{L^2(\Omega)}^2
+\varepsilon \int_0^t\|(-\Delta_D)^m u_{A,i}^\varepsilon(s)\|_{L^2(\Omega)}^2\,ds
-\sum_{j=1}^d\int_0^t\int_\Omega \Bigl(a_{ij}(s)\nabla^m u_{A,j}^\varepsilon(s)\Bigr) \nabla^m u_{A,i}^\varepsilon(s)\, dx\, ds\\
&= \sum_{j=1}^d\int_0^t\int_\Omega \Bigl(
\nabla^m\bigl(a_{ij}(s)\, u_{A,j}^\varepsilon(s)\bigr)
-a_{ij}(s)\nabla^m u_{A,j}^\varepsilon(s)\Bigr) \nabla^m u_{A,i}^\varepsilon(s)\, dx\, ds.
\end{aligned}
\]
To estimate the right hand side, we assume the domain $\Omega$ to be sufficiently smooth and employ a Leibniz rule estimate 
(\cite[Lemma 4.1]{HosonoKawashima2006}, see also \cite{DAncona2019} and the references therein),
\begin{equation}\label{eq:Leibniz}
\begin{aligned}
&\|\nabla^k(f\,g)-f\,\nabla^k g\|_{L^p(\Omega)} 
\leq C\,\Bigl(
\|\nabla f\|_{L^{q_1}(\Omega)} \|\nabla^{k-1} g\|_{L^{r_1}(\Omega)}
+\|g\|_{L^{q_2}(\Omega)} \|\nabla^k f\|_{L^{r_2}(\Omega)}
\end{aligned}
\end{equation}
with $C=C(k,p,q_1,r_1,q_2,r_2,\Omega)$ for $1\leq p,\,q_i,\,r_i\leq\infty$, $1/p=1/q_i+1/r_i$, $k\in\mathbb{N}$, $f\in W^{1,q}(\Omega)\cap W^{k,r}(\Omega)$, $g\in W^{k-1,r}(\Omega)\cap L^q(\Omega)$ and with the convention
\[
\nabla^k f=
\begin{cases}
\Delta^{k/2} f &\text{ if $k$ is even }\\
\nabla\Delta^{(k-1)/2} f &\text{ if $k$ is odd }
\end{cases}
\]
Setting $k=m$, $p=2$, 
$(q_1,r_1)=(d,\frac{2d}{d-2})$, 
$(r_2,q_2)=(2,\infty)$ in case $d<2m$,
$(r_2,q_2)=\frac{2d}{d-2m}$, $q_2=\frac{d}{m}$ in case $d>2m$
and 
\begin{equation}\label{enestA}
\begin{aligned}
&\|\nabla^m u_A^\varepsilon(t)\|_{L^2(\Omega;\mathbb{R}^d)}^2
-\|\nabla^m u_{0}\|_{L^2(\Omega;\mathbb{R}^d)}^2
+\varepsilon \int_0^t\|(-\Delta_D)^m u_A^\varepsilon(s)\|_{L^2(\Omega)}^2\,ds
\\
&\leq C \|A\|_{L^\infty(0,T;L^\infty(\Omega;\mathbb{R}^{d\times d}}
\int_0^t\|\nabla^m u_A^\varepsilon(s)\|_{L^2(\Omega;\mathbb{R}^d)}^2\, ds\\
&+ \int_0^t \|\nabla^m u_A^\varepsilon(s)\|_{L^2(\Omega;\mathbb{R}^d)}
\Bigl(\|\nabla A\|_{L^\infty(0,T;L^d(\Omega;\mathbb{R}^{d\times d}}
\|\nabla^{m-1} u_A^\varepsilon(s)\|_{L^{2d/(d-2)}(\Omega;\mathbb{R}^d)}
\\&\hspace*{5cm}
+\|\nabla^m A\|_{L^{q_2}(0,T;L^2(\Omega;\mathbb{R}^{d\times d}}
\|u_A^\varepsilon(s)\|_{L^{r_2}(\Omega;\mathbb{R}^d)}
\Bigr).
\end{aligned}
\end{equation}
Applying Cauchy-Schwarz and Young's inequality as well as 
continuity of the embeddings $H^1(\Omega)\to L^{2d/(d-2)}(\Omega)$, $H^m(\Omega)\to L^{r_2}(\Omega)$ 
and Gronwall's Lemma, we obtain
\[
\|u_A^\varepsilon(t)\|_{H^m(\Omega;\mathbb{R}^d)}
\leq C \|u_0\|_{H^m(\Omega;\mathbb{R}^d)}
\]
where $C$ depends on $m$, $d$, $\Omega$, $T$, and the norms of $A$ appearing in \eqref{enestA}, but not on $\varepsilon$.
We thus have existence of a sequence $\varepsilon_\ell\searrow0$ and of 
$\bar{u}\in L^\infty(0,T;H^m(\Omega;\mathbb{R}^d)$ such that
\[
u_A^{\varepsilon}\rightharpoonup \bar{u}\text{ in }L^\infty(0,T;H^m(\Omega;\mathbb{R}^d) \text{ as }\ell\to\infty
\]
To show that $\bar{u}$ satisfies the ODE, (note that the initial conditions are enforced on $u_A^{\varepsilon}$ and therfore on $\bar{u}$ in an $L^2(\Omega;\mathbb{R}^d)$ sense), we consider, for an arbitrary $\phi\in (C_c^\infty(0,T)\times\Omega)$ the identity
\[
\begin{aligned}
&\int_0^T\int_\Omega \bigl(\tfrac{d}{d t}\bar{u}(t,x)-A(t)\bar{u}(t,x)\bigr)\phi(t,x)\, dx\, dt\\
&= \int_0^T\int_\Omega \bigl(\tfrac{d}{d t}(\bar{u}(t,x)-u_A^\varepsilon(t,x))-A(t)(\bar{u}(t,x)-u_A^\varepsilon(t,x))
-\varepsilon (-\Delta_D)^m u_A^\varepsilon(t,x)
\bigr)\phi(t,x)\, dx\, dt\\
&= \int_0^T\int_\Omega \Bigl(-(\bar{u}(t,x)-u_A^\varepsilon(t,x))\tfrac{d}{d t}\phi(t,x)
-A(t)(\bar{u}(t,x)-u_A^\varepsilon(t,x))\Bigr.\\
&\qquad\Bigl.-\varepsilon u_A^\varepsilon(t,x)
(-\Delta_D)^m \phi(t,x)\Bigr)\, dx\, dt\\
&\to0\text{ as }\varepsilon\to0
\end{aligned}
\]

\subsection*{Proof of Proposition~\ref{prop:evolutionest_F}}

Testing \eqref{ivp_F} with $u_F$ and with $\partial_t u_F$, respectively, we obtain
\begin{equation*}
\begin{aligned}
&\sum_{i=1}^d \|u_{F,i}(t)\|_{L^2(\Omega)}^2
-\sum_{i=1}^d \|u_{0,i}\|_{L^2(\Omega)}^2
+\underline{c}_D\sum_{i=1}^d \int_0^t\|\nabla u_{F,i}(s)\|_{L^2(\Omega)}^2\,ds\\
&\leq 
\sum_{i=1}^d\int_0^t \|\velo(s)\|_{L^\infty(\Omega)}\|\nabla u_{F,i}(s)\|_{L^2(\Omega)} \|u_{F,i}(s)\|_{L^2(\Omega)}\, ds
+ \int_0^t \|A(s)+B(s)\| \sum_{j=1}^d \|u_{F,j}(s)\|_{L^2(\Omega)}^2\, ds 
\\
&\leq \frac{\underline{c}_D}{2} \sum_{i=1}^d \int_0^t\|\nabla u_{F,i}(s)\|_{L^2(\Omega)}^2\,ds
\\&\qquad
+\Bigl(\frac{2}{\underline{c}_D} \|\velo\|_{L^\infty(0,T;L^d(\Omega))}^2 
+\|A+B\|_{L^\infty(0,T;L^d(\Omega))}
\Bigr)\sum_{i=1}^d\int_0^t\|u_{F,i}(s)\|_{L^{2d/(d-2)}(\Omega)}^2\, ds,
\end{aligned}
\end{equation*}
where $H^1(\Omega)$ continuously embeds into $L^{2d/(d-2)}(\Omega)$ and
\[
\begin{aligned}
&\sum_{i=1}^d \int_0^t \|\partial_t u_{F,i}(s)\|_{L^2(\Omega)}^2\, ds
+\underline{c}_D\sum_{i=1}^d \|\nabla u_{F,i}(t)\|_{L^2(\Omega)}^2
-\|D\|_{L^\infty(\Omega)} \sum_{i=1}^d \|\nabla u_{0,i}\|_{L^2(\Omega)}^2
\\
&\leq 
\sum_{i=1}^d\int_0^t \|\velo(s)\|_{L^\infty(\Omega)}\|\nabla u_{F,i}(s)\|_{L^2(\Omega)} \|\partial_t u_{F,i}(s)\|_{L^2(\Omega)}\, ds\\
&\qquad+ \int_0^t \|A(s)+B(s)\|_{L^\infty(\Omega)} \sqrt{\sum_{j=1}^d \|u_{F,j}(s)\|_{L^2(\Omega)}^2} \sqrt{\sum_{j=1}^d \|\partial_t u_{F,j}(s)\|_{L^2(\Omega)}^2}\, ds 
\\
&\leq \frac12 \sum_{i=1}^d \int_0^t\|\partial_t u_{F,i}(s)\|_{L^2(\Omega)}^2\,ds
+\|\velo\|_{L^\infty(0,T;L^\infty(\Omega))}^2 \sum_{i=1}^d \int_0^t\|\nabla u_{F,i}(s)\|_{L^2(\Omega)}^2\,ds\\
&\qquad+\|A+B\|_{L^\infty(0,T;L^\infty(\Omega))}^2 \sum_{i=1}^d\int_0^t\|u_{F,i}(s)\|_{L^2(\Omega)}^2\, ds,
\end{aligned}
\]
Here we applied the Cauchy-Schwarz inequality in $L^2(0,T)$ and $\mathbb{R}^d$, as well as Young's inequality, and used the spectral norm as matrix norm in $\mathbb{R}^{d\times d}$.

The $H^2(\Omega)$ estimate is obtained by testing \eqref{ivp_F} with $\nabla\cdot(D\nabla \partial_t u_F)$, which yields
\[
\begin{aligned}
&\underline{c}_D\sum_{i=1}^d \int_0^t \|\nabla \partial_t u_{F,i}(s)\|_{L^2(\Omega)}^2\, ds
+\sum_{i=1}^d \|\nabla\cdot(D\nabla u_{F,i}(t))\|_{L^2(\Omega)}^2
-\sum_{i=1}^d \|\nabla\cdot(D\nabla u_{F,i}(0))\|_{L^2(\Omega)}^2
\\
&\leq 
\sum_{i=1}^d\int_0^t \|\velo(s)\|_{L^\infty(\Omega)}\|\nabla u_{F,i}(s)\|_{L^2(\Omega)} \|\nabla\cdot(D\nabla \partial_t u_{F,i}(s))\|_{L^2(\Omega)}\, ds\\
&\qquad+ \int_0^t \|A(s)+B(s)\|_{L^\infty(\Omega)} \sqrt{\sum_{j=1}^d \|u_{F,j}(s)\|_{L^2(\Omega)}^2} \sqrt{\sum_{j=1}^d \|\nabla\cdot(D\nabla \partial_t u_{F,j}(s))\|_{L^2(\Omega)}^2}\, ds, 
\end{aligned}
\]
which can be further estimated analogously to above.
Using elliptic regularity \eqref{ell-reg} and the $L^2(\Omega)$ estimate from above yields the assertion.

\subsection*{Proof of Proposition~\ref{prop:commutatorestCA}}

For any $w=(w_1,\ldots,w_d)\in H^2(\Omega)^d$ we obtain
\[
\begin{aligned}
    &([C_{\text{diff}} I,A(t)]w)_i
    =\nabla\cdot\left(D\nabla \left(\sum_{k=1}^d a_{ik}(t) w_k\right)\right)
    -\sum_{k=1}^d a_{ik}(t) \nabla\cdot\left(D\nabla w_k\right)\\
    &=\sum_{k=1}^d [C_{\text{diff}},a_{ik}(t)]w_k
    = \sum_{k=1}^d \Bigl(\nabla a_{ik}\cdot(D\nabla w_k)
    +\nabla w_k\cdot(D\nabla a_{ik}) 
    + w_k\nabla\cdot(D\nabla a_{ik})\Bigr)
\end{aligned}
\]
and likewise 
\[
    ([C_{\text{adv}}(t) I,A(t)]w)_i
    =\sum_{k=1}^d [C_{\text{adv}}(t),a_{ik}(t)]w_k
    = \sum_{k=1}^d w_k\, \velo\cdot \nabla a_{ik}
\]
Under the assumption \eqref{ass:ADV_comm},
we can estimate 
\[
\begin{aligned}
&\|\nabla a_{ik}(t)\cdot(D\nabla w_k)+\nabla w_k\cdot(D\nabla a_{ik}(t)) 
+ w_k(\nabla D \nabla a_{ik}(t) + D \nabla^2 a_{ik}(t))\|_{L^2(\Omega)}\\
&\leq 2\|\nabla a_{ik}(t)\|_{L^\infty(\Omega)} \|D\|_{L^\infty(\Omega)} \|\nabla w_k\|_{L^2(\Omega)}\\
&\qquad    + \|w_k\|_{L^{2d/(d-2)}(\Omega)}
    \Bigl(\|\nabla D\|_{L^d(\Omega)} \|\nabla a_{ik}(t)\|_{L^\infty(\Omega)}
    + \|D\|_{L^\infty(\Omega)} \|\nabla a_{ik}(t)\|_{W^{1,d}(\Omega)} \Bigr)
\end{aligned}
\]
and 
\[
\|w_k\, \velo(t)\cdot \nabla a_{ik}(t)\|_{L^2(\Omega)}
\leq \|w_k\|_{L^{2d/(d-2)}(\Omega)}
    \|\velo(t)\|_{L^d(\Omega)} \|\nabla a_{ik}(t)\|_{L^\infty(\Omega)}.
\]
Using continuity of the embedding $H^1(\Omega)\to L^{2d/(d-2)}(\Omega)$ for $\Omega\subseteq \mathbb{R}^d$, $d\geq 3$,
we can estimate
\[
\begin{aligned}
&\|[C_{\text{diff}} I,A(t)]\|_{L^2(\Omega)^d\leftarrow H^1(\Omega)^d}\\
&\leq K_{\text{diff}} \Bigl(
\|A\|_{L^\infty(0,T;W^{1,\infty}(\Omega;\mathbb{R}^{d\times d})\cap W^{2,d}(\Omega;\mathbb{R}^{d\times d}))}
+ \|D\|_{L^\infty(\Omega;\mathbb{R}^{d\times d})\cap W^{1,d}(\Omega;\mathbb{R}^{d\times d})}\Bigr)\\
&\|[C_{\text{adv}} I,A(t)]\|_{L^2(\Omega;\mathbb{R}^d)\leftarrow H^1(\Omega;\mathbb{R}^d)}\\
&\leq K_{\text{adv}} \Bigl(
\|A\|_{L^\infty(0,T;W^{1,\infty}(\Omega;\mathbb{R}^{d\times d})}
+ \|\velo\|_{L^\infty(0,T;L^d(\Omega;\mathbb{R}^d))}\Bigr)
\end{aligned}
\]
with constants $K_{\text{diff}}$, $K_{\text{adv}}$ depending only on $d$ and $\Omega$.

\subsection*{Proof of Proposition~\ref{prop:timereg_BT}}

Write a generic state as $v=(v_\perp,v_z,v_e)\in L^2(\Omega;\mathbb{R}^4)$, and recall
\[
A(t)v=\gamma\begin{pmatrix}\magind_z(t)v_\perp^\vdash-v_z\magind_\perp(t)^\vdash\\ \magind_\perp(t)^\vdash\cdot v_\perp\\ 0\end{pmatrix}\,,\qquad
Bv=\begin{pmatrix}-R_2v_\perp\\ -R_1(v_z-v_e)\\ 0\end{pmatrix}\,,\qquad
C(t)v=\begin{pmatrix}\mathcal{L}(t)v_\perp\\ \mathcal{L}(t)v_z\\ 0\end{pmatrix}\,,
\]
with $\mathcal{L}(t)=C_{\text{diff}}+C_{\text{adv}}(t)$, $C_{\text{diff}}\phi=\nabla\cdot(D\nabla\phi)$, $C_{\text{adv}}(t)\phi=-\velo(t)\cdot\nabla\phi$, and $D$ independent of time.

\medskip
\noindent\textit{Differentiability of $A(t)$.} For $h\neq0$ set
\[
q_h(t):=\frac{\magind(t+h)-\magind(t)}{h}-\pat\magind(t)\,.
\]
By $\magind\in W^{1,\infty}(0,T;L^\infty(\Omega;\mathbb{R}^3))$, $\|q_h(t)\|_{L^\infty(\Omega;\mathbb{R}^3)}\to0$ as $h\to0$ for a.e.\ $t\in(0,T)$. The candidate derivative
\[
A'(t)v=\gamma\begin{pmatrix}\pat b_z(t)v_\perp^\vdash-v_z\pat b_\perp(t)^\vdash\\ \pat b_\perp(t)^\vdash\cdot v_\perp\\ 0\end{pmatrix}
\]
satisfies
\[
\left(\frac{A(t+h)-A(t)}{h}-A'(t)\right)v
=\gamma\begin{pmatrix}q_{h,z}(t)v_\perp^\vdash-v_zq_{h,\perp}(t)^\vdash\\ q_{h,\perp}(t)^\vdash\cdot v_\perp\\ 0\end{pmatrix}\,,
\]
so that, by boundedness of multiplication by $L^\infty(\Omega)$ functions on $L^2(\Omega)$, there is $C_A=C_A(\Omega)$ with
\[
\left\|\frac{A(t+h)-A(t)}{h}-A'(t)\right\|_{L^2(\Omega;\mathbb{R}^4)\leftarrow L^2(\Omega;\mathbb{R}^4)}
\le C_A\gamma\|q_h(t)\|_{L^\infty(\Omega;\mathbb{R}^3)}\longrightarrow0\,.
\]
Hence $A'(t)$ exists for a.e.\ $t\in(0,T)$ in the operator norm $L^2(\Omega;\mathbb{R}^4)\leftarrow L^2(\Omega;\mathbb{R}^4)$, with
\[
\|A'(t)\|_{L^2(\Omega;\mathbb{R}^4)\leftarrow L^2(\Omega;\mathbb{R}^4)}\le C_A\gamma\|\pat\magind(t)\|_{L^\infty(\Omega;\mathbb{R}^3)}\,.
\]
Since $t\mapsto A(t)$ is absolutely continuous with values in this operator norm, $A(t)-A(s)=\int_s^tA'(\tau)\,\dd\tau$, whence
\[
\|A(t)-A(s)\|_{L^2(\Omega;\mathbb{R}^4)\leftarrow L^2(\Omega;\mathbb{R}^4)}
\le C_A\gamma|t-s|\,\|\pat\magind\|_{L^\infty(0,T;L^\infty(\Omega;\mathbb{R}^3))}\,.
\]

\medskip
\noindent\textit{Differentiability of $B$.} Since $R_1,R_2$ do not depend on $t$, $B$ is autonomous: $B(t)-B(s)=0$ and $B'(t)=0$, with no additional temporal regularity required on $R_1,R_2$.

\medskip
\noindent\textit{Differentiability of $C(t)$.} For $h\neq0$ set
\[
r_h(t):=\frac{\velo(t+h)-\velo(t)}{h}-\pat\velo(t)\,,
\]
so that $\|r_h(t)\|_{L^\infty(\Omega;\mathbb{R}^3)}\to0$ as $h\to0$ for a.e.\ $t$, by $\velo\in W^{1,\infty}(0,T;L^\infty(\Omega;\mathbb{R}^3))$. Since $D$, and hence $C_{\text{diff}}$, is independent of time, it cancels in the difference quotient, and the candidate derivative is
\[
C'(t)v=\begin{pmatrix}-\pat\velo(t)\cdot\nabla v_\perp\\ -\pat\velo(t)\cdot\nabla v_z\\ 0\end{pmatrix}\,,
\qquad
\left(\frac{C(t+h)-C(t)}{h}-C'(t)\right)v=\begin{pmatrix}-r_h(t)\cdot\nabla v_\perp\\ -r_h(t)\cdot\nabla v_z\\ 0\end{pmatrix}\,.
\]
For $v=(v_\perp,v_z,v_e)\in H^1(\Omega;\mathbb{R}^3)\times L^2(\Omega)$, by H\"older's inequality,
\[
\left\|\left(\frac{C(t+h)-C(t)}{h}-C'(t)\right)v\right\|_{L^2(\Omega;\mathbb{R}^4)}
\le\|r_h(t)\|_{L^\infty(\Omega;\mathbb{R}^3)}\|(v_\perp,v_z)\|_{H^1(\Omega;\mathbb{R}^3)}\,,
\]
so that
\[
\left\|\frac{C(t+h)-C(t)}{h}-C'(t)\right\|_{L^2(\Omega;\mathbb{R}^4)\leftarrow H^1(\Omega;\mathbb{R}^3)\times L^2(\Omega)}
\le\|r_h(t)\|_{L^\infty(\Omega;\mathbb{R}^3)}\longrightarrow0\,.
\]
Hence $C'(t)$ exists for a.e.\ $t\in(0,T)$ in this operator norm, with
\[
\|C'(t)\|_{L^2(\Omega;\mathbb{R}^4)\leftarrow H^1(\Omega;\mathbb{R}^3)\times L^2(\Omega)}\le\|\pat\velo(t)\|_{L^\infty(\Omega;\mathbb{R}^3)}\,,
\]
and, since the time-dependent part of $C$ is absolutely continuous in this norm, $C(t)-C(s)=\int_s^tC'(\tau)\,\dd\tau$ gives
\[
\|C(t)-C(s)\|_{L^2(\Omega;\mathbb{R}^4)\leftarrow H^1(\Omega;\mathbb{R}^3)\times L^2(\Omega)}
\le|t-s|\,\|\pat\velo\|_{L^\infty(0,T;L^\infty(\Omega;\mathbb{R}^3))}\,.
\]
Only $C(t)-C(s)$ and $C'(t)$ are estimated this way, since the stationary diffusion part $C_{\text{diff}}$ cancels in both; boundedness of $C(t)$ itself from $H^1(\Omega;\mathbb{R}^3)\times L^2(\Omega)$ to $L^2(\Omega;\mathbb{R}^4)$ is not asserted.

\medskip
\noindent\textit{Estimates for Approach~1.} For $t\in[t_n,t_n+\tau_n]$, $t_n+\tfrac13(t-t_n)-t=-\tfrac23(t-t_n)$, so $|t_n+\tfrac13(t-t_n)-t|\le\tfrac23\tau_n$, and hence
\[
\left\|A\Big(t_n+\tfrac13(t-t_n)\Big)-A(t)\right\|_{L^2(\Omega;\mathbb{R}^4)\leftarrow L^2(\Omega;\mathbb{R}^4)}
\le\tfrac23C_A\gamma\tau_n\|\pat\magind\|_{L^\infty_tL^\infty_x}\,.
\]
Likewise, $t_n+\tfrac23\tau_n+\tfrac13(t-t_n)-t=\tfrac23(t_n+\tau_n-t)$, of absolute value at most $\tfrac23\tau_n$, so that
\[
\left\|C\Big(t_n+\tfrac23\tau_n+\tfrac13(t-t_n)\Big)-C(t)\right\|_{L^2(\Omega;\mathbb{R}^4)\leftarrow H^1(\Omega;\mathbb{R}^3)\times L^2(\Omega)}
\le\tfrac23\tau_n\|\pat\velo\|_{L^\infty_tL^\infty_x}\,,
\]
while the corresponding difference for $B$ vanishes identically.

\medskip
\noindent\textit{Estimate for Approach~2.} Since $F(t)=A(t)+B+C(t)$ and $B'=0$, $F'(t)=A'(t)+C'(t)$; as $H^1(\Omega;\mathbb{R}^3)\times L^2(\Omega)\subset L^2(\Omega;\mathbb{R}^4)$, the bound on $A'(t)$ applies also on this smaller space, and
\[
\|F'(t)\|_{L^2(\Omega;\mathbb{R}^4)\leftarrow H^1(\Omega;\mathbb{R}^3)\times L^2(\Omega)}
\le C_A\gamma\|\pat\magind(t)\|_{L^\infty(\Omega;\mathbb{R}^3)}+\|\pat\velo(t)\|_{L^\infty(\Omega;\mathbb{R}^3)}\,.
\]
This establishes all estimates claimed in Proposition~\ref{prop:timereg_BT}.

\bibliographystyle{siam}
\bibliography{lit}
\end{document}